\documentclass[12pt,a4paper]{amsart}
\usepackage{amsmath,amsfonts,latexsym,amssymb,amsthm,pictexwd,dcpic}
\usepackage{latexsym}
\usepackage[pdftex]{graphicx}
\usepackage{epstopdf}
\usepackage{hyperref}
\usepackage{color}


\newcommand{\field}[1]{\mathbb{#1}}
\newcommand{\rz}{\field{R}}

\newcommand{\zz}{\field{Z}}
\newcommand{\nz}{\field{N}}

\newcommand{\I}{\mathrm{i}}

\newcommand{\arccosh}{\mathrm{arccosh}}
\newcommand{\arccot}{\mathrm{arccot}}

 \newtheorem{theorem}{Theorem}[section]
 
 \newtheorem{example}[theorem]{Example}
 \newtheorem{lemma}[theorem]{Lemma}
 \newtheorem{corollary}[theorem]{Corollary}

 \newtheorem{remark}[theorem]{Remark}

\DeclareMathOperator{\tr}{tr}

\newcommand{\genus}{\textup{g}}

\title{An Algorithm for the Computation of Eigenvalues, Spectral Zeta Functions and Zeta-Determinants on Hyperbolic Surfaces
}

\author[A. Strohmaier]{Alexander Strohmaier}
\address{Department of Mathematical Sciences,  Loughborough University,  Loughborough, Leicestershire, LE11 3TU,
UK} \email{a.strohmaier@lboro.ac.uk}

\author[V. Uski]{Ville Uski}
\address{Department of Mathematical Sciences,  Loughborough University,  Loughborough, Leicestershire, LE11 3TU,
UK}
\email{V.Uski@lboro.ac.uk}

\thanks{This work was supported by the Leverhulm grant F/00 261/Z}

\begin{document}

\maketitle

\begin{abstract}
 We present a rigorous scheme that makes it possible to compute eigenvalues 
 of the Laplace operator on hyperbolic surfaces within a given precision. The method is based
 on an adaptation of the method of particular solutions to the case of locally symmetric spaces and on
 explicit estimates for the approximation of eigenfunctions on hyperbolic surfaces by certain basis functions.
  It can be applied to check whether or not there is an eigenvalue in an $\epsilon$-neighborhood of a given number
 $\lambda>0$.
 This makes it possible to find all the eigenvalues in a specified interval, up to a given precision with rigorous
 error estimates. The method converges exponentially fast with the number of basis functions used.
 Combining the knowledge of the eigenvalues with the Selberg trace formula we are able to compute
 values and derivatives of the spectral zeta function again with error bounds. As an example we calculate the spectral determinant
 and the Casimir energy of the Bolza surface and other surfaces.
 \end{abstract}

\section{Introduction}

Let $(M,g)$ be a compact Riemannian manifold of dimension $n$ and let $\Delta$ be the (positive) Laplace operator
on functions, which in local coordinates is given by
$$
 \Delta = - \sum_{i,k=1}^n \frac{1}{\sqrt{|g|}} \frac{\partial}{\partial x_i} \sqrt{|g|} g^{ik} \frac{\partial}{\partial x_k}.
$$
Here $|g|$ denotes the determinant of the metric tensor and  $g^{ik}$ are the components of the dual metric on the cotangent bundle.
This operator is formally self-adjoint on the space of smooth functions with inner product
$$
 \langle f_1, f_2 \rangle = \int_M f_1(x) \overline{f_2}(x) \sqrt{|g|} dx_1, \cdots dx_n
$$
defined by the Riemannian measure $\sqrt{|g|} dx_1 \cdots dx_n$.
Then $\Delta$ extends to a self-adjoint operator $L^2(M) \supset H^2(M) \to L^2(M)$
with compact resolvent. This means there exists an orthonormal basis $\{\phi_j \mid j \in \mathbb{N}_0\}$
in $L^2(M)$ consisting of eigenfunctions
$$
 \Delta \phi_j = \lambda_j \phi_j
$$
which we assume to be ordered such that $0=\lambda_0 \leq \lambda_1 \leq \lambda_2 \leq \ldots$.

It is a classical problem to compute the eigenvalues and the eigenfunctions of the Laplace operator
on a given manifold. Its solution allows complete control over the functional calculus and thus over solutions
to the heat equation, the wave equation, or the Schr\"odinger equation.

In this paper we show how an adaptation of the method of particular solutions can be used to compute eigenvalues
on manifolds with a high accuracy, beyond of what can currently be achieved by finite element methods
or boundary element methods.
We will focus primarily on the case of 
two dimensional oriented surfaces of constant curvature. These are the simplest examples of manifolds with
non-trivial spectral geometry. 
They are topologically classified by their genus. By the Gauss-Bonnet theorem the curvature $\kappa$ 
of a surface of genus $\genus$ can always be normalized 
to be $\kappa=\mathrm{sign}(\genus-1)$ by multiplying the metric by a positive constant.
This divides oriented surfaces of constant curvature into three categories
\begin{itemize}
 \item $\kappa=+1, \genus=0$: up to isometries there is only one such manifold: the round sphere $S^2$. The
 spectrum of the Laplace operator on the round sphere is of course well known and the spherical harmonics provide a
 basis consisting of eigenfunctions.
 \item $\kappa=0, \genus=1$: such manifolds are isometric to flat tori which are obtained as quotients of $\mathbb{R}^2$ by 
  co-compact lattices. The moduli space of equivalence classes of flat metrics on a given topological torus
  is the modular surface $\mathrm{SL}(2,\mathbb{Z})\backslash \mathbb{H}$, where 
  $\mathbb{H}=\{x+\mathrm{i} y \in \mathbb{C} \mid y>0\}$ is the upper half space. For a surface $\mathbb{R}^2/L$
  a basis consisting of eigenfunctions is obtained by taking the Fourier modes associated with points in the dual lattice $L^*$
  and the corresponding eigenvalues are the squares of the norms of these points.
 \item $\kappa=-1, \genus>1$: these are the so called hyperbolic surfaces. They can be obtained either as quotients
 $\Gamma \backslash \mathbb{H}$ by co-compact hyperbolic lattices in $\mathrm{SL}(2,\mathbb{R})$ or by
 glueing hyperbolic pairs of pants. The moduli space of hyperbolic metrics on a given topological surface
 is a quotient of the Teichm\"uller space $\mathcal{T}_{\genus}$ by a discrete group, the mapping class group.
 The Teichm\"uller space for a genus $\genus$ surface has dimension $6 \genus -6$ so that a constant curvature metric
 on a topological surface may be given by specifying $6 \genus -6$ parameters.
 Unlike in the previous two cases the spectrum of the Laplace operator of a hyperbolic surface can not currently
 be explicitly computed.
\end{itemize}

The main purpose of this article is to establish an algorithm to compute the eigenvalues and eigenfunctions
of the Laplace operator on a hyperbolic surface up to a certain precision and with mathematically rigorous error bounds.
Once this is achieved for a number of eigenvalues it is then possible to compute values of the spectral zeta function
and the spectral determinant of the Laplace operator on such surfaces. We demonstrate how this can be done using the Selberg
trace formula in such a way
that the error can again be bounded rigorously.

The method we use is adapted to a pants decomposition of the surface. For a particular
basis of functions we establish various bounds and estimates that prove that the method of particular solutions
can be applied to an interval and yields all eigenvalues in that interval. In particular we prove a novel estimate that allows us
to show non-existence of eigenvalues in a certain interval.

\subsection{Organization of the article and results}
In section \ref{mps} we describe how an adaptation of the method of particular solutions can be applied to manifolds.
The method uses a finite dimensional space of test functions that solve the eigenvalue equation with eigenvalue $\lambda$ 
on open submanifolds that are glued together along co-dimension one hypersurfaces. The  glueing conditions for the eigenfunctions,
continuity of the eigenfunctions and their normal derivatives, is measured in terms of the differences of the function
values and normal derivatives along the hypersurfaces. For a test function $\phi$ we introduce in Equation (\ref{functionaldef}) 
the number $F_{-1/2,-3/2}(\phi)$ that measures these differences in suitable Sobolev norms.
We prove that if $\| \phi \|_{L^2(M)}=1$ and $F_{-1/2,-3/2}(\phi)$ is small then $\lambda$ must be close to an eigenvalue. 
Our Theorem  \ref{main01} together with the subsequent estimates on the constants gives a quantitative version of this statement
for hyperbolic surfaces that are glued along geodesic segments.

In section \ref{sec:surfaces} we review the construction of hyperbolic surfaces from pairs of pants.
This gives an explicit parametrization of Teichm\"uller space in terms of Fenchel-Nielsen coordinates and shows 
that hyperbolic surfaces can be cut open along geodesic segments 
into pairs of pants and subsequently into subsets of hyperbolic cylinders.

In section \ref{sec:MPS_on_surface} we describe a set of basis functions for a surface of genus $\genus$
that is decomposed into pairs of pants. We prove that true eigenfunctions can be approximated by linear combinations
of this set of basis functions with explicit bounds on the error. Section \ref{algo}  deals with the construction of 
$m \times k$ matrices $\mathbf{B}^0_\lambda$, $\mathbf{B}_\lambda$ and $\mathbf{C}_\lambda$ with the following properties. 
\begin{enumerate}
\item
 The distance of $\lambda$ to the spectrum can be bounded from above in terms of the first singular value
 $\sigma_1(\mathbf{B}^0_\lambda)$ of $\mathbf{B}^0_\lambda$ and its singular vector (Theorem \ref{mainm01}).
\item The smallest relative singular value $\sigma_1(\mathbf{B}_\lambda,\mathbf{C}_\lambda)$ is bounded from above by
$$c_1(k,\lambda) + c_2(\lambda) \mathrm{dist}(\mathrm{spec}(\Delta),\lambda),$$  where $c_1(k,\lambda)$ and $c_2(\lambda)$
are explicitly computable constants and $c_1$ is exponentially decaying in $k$ (Theorem \ref{maindist}).
\end{enumerate}
Whereas the first property allows to prove that an eigenvalue is in a certain interval, the second property
allows to determine intervals in which there are no eigenvalues. Both estimates together
can be used to find all eigenvalues in a specified interval.
We demonstrate that these matrices can be computed within a given precision and show that the singular values can be
bounded from above and below using interval arithmetics.
As a proof of concept we implemented our method in Fortran and in Mathematica. This resulted in programs that allow
to compute eigenvalues rather accurately for a surface of genus $\genus$ with given Fenchel-Nielsen
coordinates.
A Mathematica program was used to compute the first eigenvalues of the Bolza surface
with extremely high accuracy.

The spectral zeta function $\zeta_{\Delta}(s)$ is defined as the meromorphic continuation of the function
$$
 \zeta_{\Delta}(s)=\sum_{i=1}^\infty \lambda_i^{-s},
$$
where $(\lambda_i)_{i \in \nz}$ are the non-zero eigenvalues of the Laplace operator repeated according to their multiplicities.
The (zeta-regularized) spectral determinant $\mathrm{det}_\zeta(\Delta)$ is defined by
$\log \mathrm{det}_\zeta(\Delta)= -\zeta'_{\Delta}(0)$ (zero is not a pole of the meromorphic continuation).
Since the meromorphic continuation is contructed from the full spectrum it is a priori not enough to
know a finite part of the spectrum to compute the spectral determinant up to a certain accuracy.
However, we show in section \ref{zeta} that the Selberg trace formula may be used in conjunction with the list of eigenvalues
up to a certain threshold to calculate values of the spectral zeta function, in particular the Casimir energy and the spectral determinant,
up to a certain precision depending on that threshold.
The spectral determinant 
for the Laplace operator on hyperbolic surfaces is of particular importance.
In the seminal paper \cite{MR960228} Osgood, Phillips and Sarnak showed that the spectral determinant is maximized
at the hyperbolic metric in each conformal class of a given volume. Thus, understanding the extremal properties of the
determinant as a function on the space of metrics of fixed volume in dimension two is equivalent to the understanding
of the spectral determinant as a function on the Teichm\"uller space of hyperbolic metrics. Of course values of the determinant
for non-hyperbolic surfaces can be computed  from the value for the corresponding uniformized
hyperbolic surface using the Polyakov formula.

Finally, section \ref{zetaexa} contains of collection of examples of interesting surfaces of genus two and three 
for which we computed the spectral determinant and the value of the spectral zeta function at the point $-1/2$.
The set of examples contains the isolated surfaces of genus two with large symmetry group which are well known to be critical points
for any value of the spectral zeta function or the spectral determinant (see e.g. \cite{MR2452638}).
We conjecture that the value of the spectral
determinant is maximized at the Bolza surface and we compute the
value of the spectral determinant rather accurately. Such a conjecture would imply that the global maximum
of the spectral determinant as a function on the space of all metrics of fixed volume is attained at this point.
The appendix contains explicit estimate for various resolvents as well as explicit estimates of the
$L^\infty$ norm of eigenfunctions and derivatives of eigenfunctions. These estimates are needed to make the constants
in our estimates explicit but they may also be interesting in their own right.

\subsection{Discussion} 
 Eigenvalues of hyperbolic surfaces of genus two have been calculated in the physics literature by Aurich and Steiner
 in \cite{Aurich:1989} using the finite element method, and in \cite{MR1214552} using the boundary element method.
 In both cases the authors relied on the realization of the surface by geodesic octagons.
 It is quite interesting that the Hadamard-Gutzwiller model discussed in \cite{Aurich:1989}
 is actually the same as the Bolza surface. Using the group action and the decomposition into geodesics
 triangles Ninnemann \cite{Ninnemann:1995} computed the eigenvalues of the regular geodesic octagon
 with every second side identified. The resulting surface, which he also refers to as the Gutzwiller octagon,
 does however not coincide with the Bolza surface. Its Fenchel Nielsen coordinates are given in Section
 \ref{regoct}.
 There as well as in other parts of the physics literature the Bolza surface is referred to as the regular octagon.
 
 The method of particular solutions is based on an article by Fox, Henrici and Moler
 \cite{Fox:1967}. It was subsequently further developed and revived by Betcke and Trefethen \cite{Betcke:2005}
 to achieve high accuracy in eigenvalue computations on domains with Dirichlet boundary conditions.
 Further versions, including
 domain decompositions and matching of boundary data including the normal derivative were described in
 \cite{MR2652084} and \cite{MR2337576}. 
 Explicit error bounds for the eigenvalues in the method of particular solutions for domains in 
 $\mathbb{R}^n$ were given in \cite{1968} and further improved 
 in by Alex Barnett \cite{MR2519590}. We would also like to refer to the latter article for further background
 and literature on the method of particular solutions. Our estimate for the error bound does not use 
 \cite{1968} but is instead based on the more refined information contained in the resolvent of the Laplace operator
 on the manifold. 
 Thus, for generic eigenfunctions it is expected to give an improvement of the order of the square root of the eigenvalues.
 Estimates that provide such an improvement without this genericity assumption have recently been obtained
 for domains in $\rz^n$
 in \cite{MR2519590} and \cite{MR2812557}.

 Finally, for hyperbolic surfaces with cusps Hejhal (\cite{Hejhal:92AMS} and \cite{Hejhal:99}) introduced
 a  method to compute Maass cusp forms which is based on the group action and
 the expansion of the embedded eigenvalues on the cusp. After most of the work in this paper was completed we
 learned that Booker, Str\"ombergsson and Venkatesh \cite{MR2249995} have recently used
 the pre-trace formula together with the Taylor expansion of boundary data of quasi-modes
 to rigorously verify embedded eigenvalues for the modular surface $SL(2,\mathbb{R})\backslash \mathbb{H}$.
 Their way of certifying eigenvalues is in spirit similar to the method we use to prove that our computed 
 eigenvalues are accurate within the error bounds. Apart from the fact that their method applies to a different
 geometric situation (surfaces with at least one cusp) the method of bounding the error
 is different when it comes to technical detail: for example in \cite{MR2249995} it is estimated in terms of the $L^\infty$-norm
 of the boundary data instead of the $L^2$-norm.
 It would be very interesting in the future to combine these ideas. In particular a modification of our estimates
 applies to hyperbolic surfaces with cusps. The problem of adjusting the step size in the search for Maass cusp forms
 on $SL(2,\mathbb{Z})\backslash \mathbb{H}$ in such a way that
 no eigenvalues are missed could be tackled in this way.
 
 Formulae for the spectral determinant based on the length spectrum were given in the mathematics literature by 
 Fried \cite{MR837526} and by  Pollicott and Rocha \cite{MR1474163}, and in the physics literature by Aurich and Steiner
 \cite{MR0888533,AS:1992}.
In \cite{MR1474163}, exponential convergence was proved
and a numerical algorithm was established for the case of surfaces of genus $2$ in $mw$-Fenchel Nielsen
coordinates without twisting. We are not able to confirm the numerical values obtained
in \cite{MR1474163} in the three examples there but obtain quite different values. We believe that the part of the
length spectrum computed in \cite{MR1474163} was not sufficient to obtain the correct values.
In contrast to the other methods employed our approach to compute the spectral determinant and 
the values for the zeta function allows for explicit error
estimates even if the length spectrum is unknown.

\section{The method of particular solutions on manifolds} \label{mps}

The method of particular solutions is a method to approximate eigenvalues and eigenfunctions
of a differential operator. We will consider here the case of the Laplace operator 
$\Delta: C^\infty(M) \to C^\infty(M)$ acting on functions on a compact oriented Riemannian manifold $M$.

To fix notations suppose that $\Gamma \subset M$ is a closed subset which is the finite
union of oriented compact codimension one submanifolds $\Gamma_\alpha$, possibly with boundary.
We will assume here for simplicity that the $\Gamma_\alpha$ intersect at most at boundary points, i.e.
$\Gamma_\alpha \cap \Gamma_\beta \subset \partial \Gamma_\alpha \cap \partial \Gamma_\beta$.
Removing $\Gamma$ from $M$ results in an open manifold $M \backslash \Gamma$ that is the 
interior of a manifold with piecewise smooth
boundary. This manifold might however have corners or other singularities. 
We would like to define the class of functions that are "smooth up to the boundary" on 
$M \backslash \Gamma$. In order to do this let us be more precise about how the manifold with piecewise
smooth boundary is constructed from $M$ and $\Gamma$.
The space $M \backslash \Gamma$ equipped with the geodesic distance is not complete and we denote
by $\overline{M \backslash \Gamma}$ its abstract metric completion. Since $M$ is complete
there is a natural continuous surjection $\pi: \overline{M \backslash \Gamma} \to M$. 
We define a function $f \in C(\overline{M \backslash \Gamma})$ to be smooth if and only if
for every point $x \in \overline{M \backslash \Gamma}$ there exists an open neighborhood
$\mathcal{U}$ such that there exists a smooth function $g$ on $M$ with $f|_\mathcal{U}=\pi^*(g)|_\mathcal{U}$.
The corresponding sheaf of smooth functions endows $\overline{M \backslash \Gamma}$ with the structure
of a smooth manifold with piecewise smooth boundary.
We write $C^\infty(\overline{M \backslash \Gamma})$ for the space of smooth functions
on $\overline{M \backslash \Gamma}$.

Of course an element $\phi \in C^\infty(\overline{M \backslash \Gamma})$ can be understood as a function in
$L^\infty(M)$ as the boundary of $\overline{M \backslash \Gamma}$ has zero measure and the interior
of $\overline{M \backslash \Gamma}$ can be identified with $M \backslash \Gamma$.
We will now make this identification without further mention. 
Since each $\Gamma_\alpha$ is oriented there is a natural unit normal vector field $n_\alpha$ to $\Gamma_\alpha$
and a closed tubular neighborhood of $\Gamma_\alpha$ diffeomorphic to $\Gamma_\alpha \times [-\epsilon,\epsilon]$.
A function $\phi \in C^\infty(\overline{M \backslash \Gamma})$ therefore has two boundary values on $\Gamma_\alpha$,
a right boundary value $\phi^+_\alpha$ and a left boundary value $\phi^-_\alpha$. Both are smooth functions on $\Gamma_\alpha$.
Of course the function $\phi\in C^\infty(\overline{M \backslash \Gamma})$ is continuous on $M$ if and only if for all indices $\alpha$
$$
  \forall x \in \Gamma_{\alpha}: \qquad \phi^+_\alpha(x)=\phi^-_\alpha(x).
$$
Similarly,  $\phi \in C^1(M)$ if and only if for all $\alpha$
\begin{gather*}
 \forall x \in \Gamma_{\alpha}: \qquad \phi^+_\alpha(x)=\phi^-_\alpha(x),\quad \textrm{and}\\
 \forall x \in \Gamma_{\alpha}: \qquad (n_\alpha \phi)^+(x)= (n_\alpha \phi)^-(x).
\end{gather*}
Here, $(n_\alpha \phi)^\pm$ denotes the right and left limits of the normal derivative of $\phi$.

Let us define $D_{\alpha} \phi$ to be the function on each manifold $\Gamma_\alpha$ given by 
$\phi^+_\alpha-\phi^-_\alpha$ and let $D_{n_\alpha} \phi$ be the function given by
$(n_\alpha \phi)^+ - (n_\alpha \phi)^-$. Let $D \phi$ and $D_n \phi$
be the corresponding functions in $L^\infty(\Gamma)$. As the 
boundary of $\Gamma$ has zero measure in $\Gamma$ these functions are well defined and they are smooth
on each $\Gamma_\alpha$.
 
Let $s,t \leq 0$. The functional $F^\Gamma_{s,t}(\phi)$ defined by
\begin{gather} \label{functionaldef}
 F^\Gamma_{s,t}(\phi):=(\Vert D  \phi \Vert_{H^s(\Gamma)}^2 + \Vert D_n \phi \Vert_{H^t(\Gamma)}^2)^{1/2},
\end{gather}
measures the continuity of the function and its normal derivatives in different Sobolev
norms.
Here the $H^s$ norm on $\Gamma$ is defined as
$$
 \Vert f \Vert_{H^s(\Gamma)}^2 = \sum_{\alpha} \Vert f_\alpha \Vert_{H^s(\Gamma_\alpha)}^2.
$$

Let now as above $\phi \in C^\infty(\overline{M \backslash \Gamma})$ and let 
$\chi \in C^\infty(\overline{M \backslash \Gamma})$ be defined by $\chi(x)=(\Delta-\lambda) \phi(x)$
for $x \in M \backslash \Gamma$.
Suppose that $f \in C^\infty(M)$ is any test function. Then, by Green's formula (see e.g. Prop. 4.1 in \cite{MR1395148}),
\begin{gather*}
 \int_M \phi ((\Delta-\lambda) f)(x) d\mu(x) = \\ = \int_M  ((\Delta-\lambda)\phi)(x) f(x) d\mu(x) +
 \int_{\partial \overline{M \backslash \Gamma}} \left( \mathbf{n} \phi(x) f(x) -   \phi(x) \mathbf{n} f(x) \right) d\nu(x) = \\=
 \int_M  \chi(x) f(x) d\mu(x) + \int_\Gamma (D_n \phi)(x) f(x) d\nu_\Gamma(x) - \int_\Gamma (D \phi)(x) (\mathbf{n} f)(x) d\nu_\Gamma(x),
\end{gather*}
where $\mathbf{n}$ denotes the unit normal vector field along $\Gamma$, $d\mu(x) = \sqrt{g} dx$ is the measure
induced by the Riemannian metric on $M$, and $d\nu(x),d\nu_\Gamma(x)$ are the measures on 
$\partial \overline{M \backslash \Gamma}$ and $\Gamma$ induced by the Riemannian metrics there.

If we think of $\phi$ as a distribution on $M$, then the above simply means that we have in the distributional sense
 $$
  (\Delta-\lambda) \phi = \chi + D \phi \; \delta'_\Gamma  +  D_n \phi \; \delta_\Gamma ,
 $$
 where $\delta_\Gamma$ is the Dirac delta distribution on $\Gamma$ 
 and $\delta'_\Gamma=\mathbf{n} \delta_\Gamma$ is its normal derivative so that for a test function
 $f \in C^\infty(M)$ the above distributions are given by
 $$
   \left( D_n \phi\;\delta_\Gamma \right) (f) = \int_\Gamma (D_n \phi)(x) f(x) d\nu_\Gamma(x)
 $$
 and
 $$
 \left( D \phi (\delta'_\Gamma) \right)(f) = -\int_\Gamma (D \phi)(x) (\mathbf{n} f)(x)  d\nu_\Gamma(x).
 $$
 Note that the distribution 
 $$
  h = D \phi \;\delta'_\Gamma  +  D_n \phi \; \delta_\Gamma
 $$
 is in $H^{-2}(M)$. 
 Denote by $\tilde F^\Gamma(\phi)$ its $H^{-2}$-norm, where the $H^s$-norm is defined by
 $\Vert \psi \Vert_{H^s(M)}=\Vert (1+\Delta)^{s/2} \psi \Vert_2$.
 
By the trace theorem for Sobolev spaces (see e.g. Prop. 4.5 in \cite{MR1395148}) the maps
$f \mapsto f|_\Gamma$ and $f \mapsto (\mathbf{n} f)|_\Gamma$ extend to bounded maps from $H^2(M) \to H^{3/2}(\Gamma)$
and $H^2(M) \to H^{1/2}(\Gamma)$ respectively. Hence, by duality,
 $$
  \Vert D \phi \; \delta'_\Gamma\Vert_{H^{-2}(M)} \leq \tilde C \Vert D \phi \Vert_{H^{-\frac{1}{2}}(\Gamma)}
 $$
 and
 $$
  \Vert D_n \phi \; \delta_\Gamma \Vert_{H^{-2}(M)} \leq \tilde C' \Vert D_n \phi \Vert_{H^{-\frac{3}{2}}(\Gamma)},
 $$
 for some constants $\tilde C,\tilde C' >0$. 
 Thus, there exists a constant $C>0$
 depending only on the geometry of $M$ and $\Gamma$ such that
 \begin{gather} \label{const}
  \tilde F^\Gamma(\phi) \leq C \cdot F^\Gamma_{-\frac{1}{2},-\frac{3}{2}}(\phi).
 \end{gather}
 A constant $C$ in this estimate can be explicitly found for a given manifold, for example by using gluing functions
 and the Fourier transform. We will illustrate in section \ref{sec:MPS_on_surface}
 how to get a bound on $C$ for a given hyperbolic surface when $\Gamma$ consists of geodesic segments 
 (cf. (\ref{sbound1}) and (\ref{sbound2})).
 
The method of particular solutions relies on the following simple observation.
\begin{theorem} \label{main01}
 Suppose that $\phi \in C^\infty(\overline{M \backslash \Gamma}) \subset L^2(M)$ with $\Vert \phi \Vert_2 = 1$ such that  
 \begin{gather*}
   \forall x \in M \backslash \Gamma: \quad (\Delta - \lambda) \phi(x) = \chi(x),\\
   \chi \in L^2(M), \quad 0 \leq \Vert \chi \Vert \leq \eta.
 \end{gather*}
 If $\tilde F^\Gamma (\phi)< \epsilon <1$, then
  there is an eigenvalue of $\Delta$ in the interval 
 $$
  [\lambda - \frac{(1+\lambda)\epsilon+\eta}{1-\epsilon},\lambda + \frac{(1+\lambda)\epsilon+\eta}{1-\epsilon}].
 $$
\end{theorem}
\begin{proof}
 As before let 
 $$
  h = D \phi \; \delta'_\Gamma  +  D_n \phi \; \delta_\Gamma
 $$
 Then $g=(\Delta+1)^{-1} h \in L^2(M)$ and by our assumptions
  $$
  \Vert g \Vert_2 \leq \epsilon.
 $$
 Now note that
 $$
  (\Delta-\lambda)(\phi-g) = \chi + (1+\lambda) g.
 $$
 Of course
 $$
  \Vert \chi + (1+\lambda) g \Vert_2 \leq \eta + (1+\lambda) \Vert g\Vert_2
 $$
 and
 $$
  \Vert \phi-g \Vert_2 \geq 1- \Vert g \Vert_2
 $$
 as $\Vert \phi \Vert_2=1$.
 From this it immediately follows that in case $(\Delta-\lambda)^{-1}$ exists its operator norm
 is bounded from below by
 $$
  \frac{1- \Vert g \Vert_2}{\eta + (1+\lambda) \Vert g\Vert_2}.
 $$
 This implies the theorem since the operator norm of the resolvent at the point $\lambda$ is bounded from above by
 the inverse of the distance of $\lambda$ to the spectrum.
\end{proof}

\begin{remark} \label{remrem}
 The statement of the previous theorem also applies to situations where eigenvalues might be close
 to each other or have high multiplicities. Not to overload notation we state this here only as a remark.
 If there is an orthonormal set $(\phi_i)_{i=1\ldots,k}$ in $L^2(M)$ such that for each of the $\phi$
 the conditions of Theorem \ref{main01} hold, then the same proof shows that there are
 at least $k$ eigenvalues (counting multiplicities) in a small interval around $\lambda$.
\end{remark}

To apply the method of particular solutions one constructs normalized functions $\phi$ which are eigenfunctions
of the Laplace operator on each of the $M_i$ with eigenvalue $\lambda$ and for which  $F^\Gamma_{-\frac{1}{2},-\frac{3}{2}}(\phi)$ is small.
The above theorems then show that $\lambda$ must be close to an eigenvalue and $\phi$
is close to an eigenfunction.
We will show in the following how this can be done for oriented hyperbolic surfaces of genus $\genus$.

\section{Hyperbolic surfaces}
\label{sec:surfaces}

A hyperbolic surface is a $2$-dimensional orientable Riemannian manifold endowed with a metric
of constant negative curvature equal to $-1$.  Such surfaces can be obtained by factorizing the
hyperbolic plane $\mathbb{H}=\{x+\I y \in \mathbb{C} \mid y>0\}$ with metric $y^{-2}(dx^2+dy^2)$
by a discrete co-compact hyperbolic subgroup $\Gamma$ of $SL(2,\mathbb{R})$.
The isometric action of $SL(2,\mathbb{R})$ is given by fractional linear transformations
$$
 \left( \begin{matrix} a & b \\ c & d \end{matrix} \right )  z =  \frac{a z +b}{c z +d}.
$$
Another way to obtain a surface of genus $\genus$ with hyperbolic metric is to
glue $2 \genus-2$ hyperbolic pair of pants.  A hyperbolic pair of pants is a genus $0$
surface with boundary $S^1 \dot \cup S^1 \dot \cup S^1$ equipped with a metric
of constant negative curvature $-1$ such that the boundary curves are
geodesics (see Figure \ref{pants}). We will also refer to such a surface as a $Y$-piece.

\begin{figure}[htp] 
\centering
\includegraphics*[width=3cm]{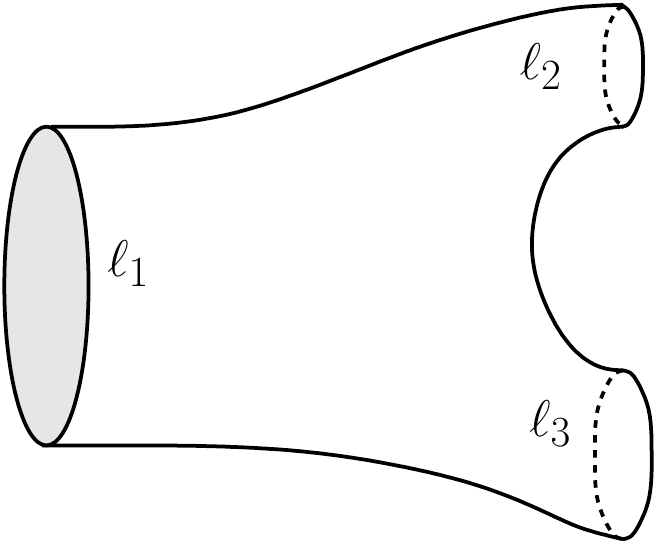}
\caption{$Y$-piece with boundary geodesics}\label{pants}
\end{figure}

Such hyperbolic pairs of pants are, up to isometry, uniquely determined by the
length of their boundary geodesics $(\ell_1,\ell_2,\ell_3)$. Any hyperbolic
surface can be decomposed into pairs of pants by cutting along $3\genus-3$
non-intersecting simple geodesics on the surface (see Figure \ref{genus3figure}).  This results in $2 \genus
-2$ pairs of pants. 

\begin{figure}[htp]
\centering 
\includegraphics*[width=10cm]{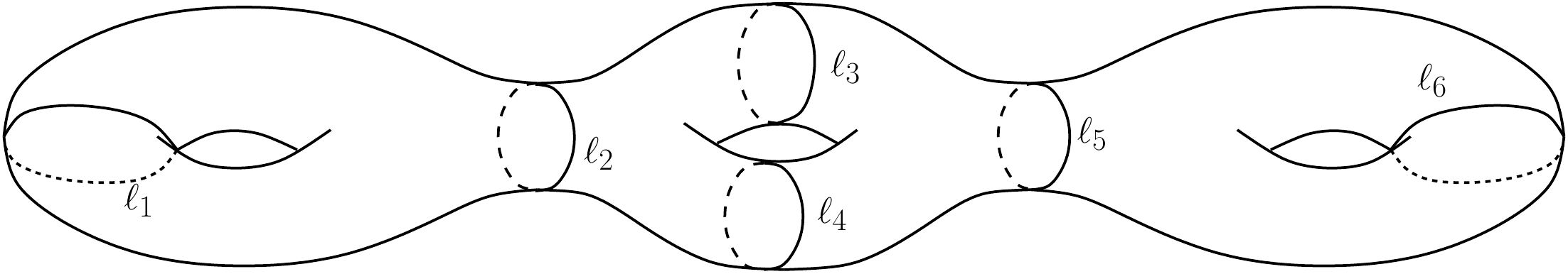}
\caption{Pants decomposition of a surface of genus three into four $Y$-pieces} \label{genus3figure}
\end{figure} 

A pair of pants can be glued from a right angled geodesic octagon in the upper half space (see Figure \ref{octi})
by identifying the sides $b$ and $h$, as well as $e$ and $g$ respectively. As indicated in the figure the geodesic
octagon is constructed from two right angled geodesic hexagons.

\begin{figure}[htp] 
\centering 
\includegraphics*[width=5cm]{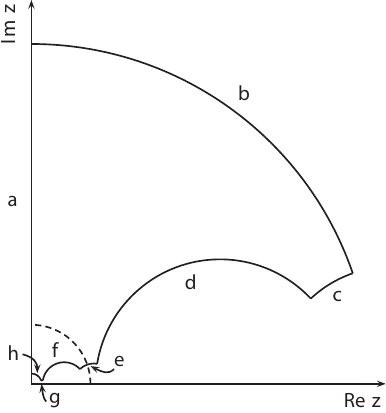}
\caption{$Y$-piece glued from a right angled geodesic octagon}\label{octi}
\end{figure} 

Note that identification of the sides $b$ and $h$ only yields a subset of a hyperbolic cylinder so that
every pair of pants can be constructed from a hyperbolic cylinder by cutting and gluing.

\section{The method of particular solutions on a surface of genus $\genus$}
\label{sec:MPS_on_surface}

\subsection{Explicit error estimates for hyperbolic surfaces} \label{s:4.1}

The first ingredient in our algorithm is an estimate on the constant in 
the estimate (\ref{const}) for $\tilde F^\Gamma$ in the particular case
when $M$ is a hyperbolic surface and $\Gamma$ is a finite union of geodesic segments.
To start let us assume that $\Gamma$ consists of exactly one geodesic segment. The general
case of a finite union of geodesic segments will later be reduced to this case.
We can assume that $M$ is realized as a quotient of the upper half plane $\mathbb{H}$ by some
co-compact hyperbolic subgroup in $\mathrm{SL}(2,\rz)$
and that $D \subset \mathbb{H}$ is an (open) fundamental domain
whose boundary is a geodesic polygon. This means we can identify functions on $M$
with functions on  $\mathbb{H}$ that are invariant under the action of the group.
We can also assume that the geodesic segment $\Gamma$
is a segment on the imaginary axis.
Then there is a tesselation of $\mathbb{H}$ by translates $D_i$
of $D$.
Let $L>0$ and suppose that $D_1,\ldots,D_N$ is a finite number of translates
of $D$ such that $\mathrm{dist}(x,D) \leq L$ implies that
$x$ is contained in the closure of $\cup_{i=1}^N D_i$.
If $L$ is chosen as the smallest distance between non-adjacent sides of $D$ then $N-1$ would be
the minimal number of fundamental domains of $D$ needed to cover an open neighborhood of
$\overline D$ (and thus the number of neighboring fundamental domains sharing either an edge or a corner 
with $D$).
Let $\chi: \mathbb{H} \to [0,1]$ be smooth compactly supported function which is equal to
one in an open neighborhood $O$ of $D$ with support contained in
$\tilde D = \overline{\cup_{i=1}^N D_i}$.
Now suppose that $h$ is a distribution on $M$. Then,
\begin{gather*}
 \| h \|_{H^{-2}(M)} = \sup_{0 \not=f \in C^\infty(M)} \frac{|h(f)|}{\|(\Delta+1)f\|_{L^2(M)}}=\\
 =\sup_{0 \not= f \in C^\infty(M)} \frac{|h(\chi f)|}{\|(\Delta+1) \chi f\|_{L^2(D)}}
\end{gather*}
Using
\begin{gather*}
 (\Delta+1)(\chi f) = f (\Delta \chi) + \chi  (\Delta+1) f + 2 (\nabla \chi, \nabla f), 
\end{gather*}
where the inner product is with respect to the metric, we get
\begin{gather*}
 \|(\Delta+1) \chi f\|_{L^2(\tilde D)}
 \leq \|\chi (\Delta+1) f\|_{L^2(\tilde D)} + \|f \Delta \chi\|_{L^2(\tilde D)}+
 2 \|\langle \nabla \chi, \nabla f \rangle \|_{L^2(\tilde D)}.
\end{gather*}
Of course
\begin{gather*}
  2 \|\langle \nabla \chi, \nabla f \rangle \|_{L^2(\tilde D)} \leq 
  2 \| \nabla \chi \|_\infty \| \nabla f \|_{L^2(\tilde D)}=\\=
  2 \| \nabla \chi \|_\infty \sqrt{\langle f, \Delta f \rangle_{L^2(\tilde D)}} \leq
   \| \nabla \chi \|_\infty \|(\Delta+1) f\|_{L^2(\tilde D)}.
\end{gather*}
Since $f$ is invariant under the group action 
$$
 \|(\Delta+1) f\|_{L^2(\tilde D)}=\sqrt{N} \|(\Delta+1) f\|_{L^2(D)},
$$
and therefore we obtain 
\begin{gather*}
 \|(\Delta+1) \chi f\|_{L^2(\tilde D)} \leq \tilde C \|(\Delta+1) \chi f\|_{L^2(D)}
\end{gather*}
with 
$$
 \tilde C = \sqrt{N} \left( 1 +  \| \nabla \chi \|_\infty + \| \Delta \chi \|_\infty \right).
$$
Consequently,
\begin{gather} \label{sob}
\| h \|_{H^{-2}(M)}  \leq \tilde C \sup_{f \in C^\infty(M)} \frac{|h(\chi f)|}{\|(\Delta+1) (\chi f)\|_{L^2(\tilde D)}}
\end{gather}
To continue we use a particular coordinate system on the upper half space.
Namely, we identify the upper half space with $(-\pi/2,\pi/2) \times \mathbb{R} $
using coordinates $(\varphi,t)$ and the metric $(\cos{\varphi})^{-2}(d\varphi^2+dt^2)$ such that
$\varphi=0$ coincides with the imaginary axis and the coordinate system is centered at the point $\I$. 
Note that $(\varphi,t)$ is related to Fermi coordinates
$(\rho,t)$ by $(\cosh{\rho},t)=((\cos{\varphi})^{-1},t)$.
We assume that in these coordinates $\Gamma$ is identified with the segment $\{0\} \times [0,\ell]$
so that $h$ has the form
\begin{gather*}
 h(f) = h_1(f) +h_2(f)=\int_0^\ell F_1(t) f(0,t) dt - \int_0^\ell  F_2(t) \left(\frac{d}{d\varphi }f \right)(0,t) dt.
\end{gather*}
Moreover in these coordinates $\Delta$ has the form $(\cos^{2} \varphi) \Delta_e$ where
$$\Delta_e = -\left( \frac{\partial^2}{\partial \varphi^2} + \frac{\partial^2}{\partial t^2}\right)$$
is the Euclidean Laplace operator in coordinates $(\varphi,t)$.
Then we have
\begin{gather*}
 \|(\Delta+1) (\chi f)\|_{L^2(\tilde D)}^2 = \\=
 \| |\cos{\varphi}| \Delta_e (\chi f) \|^2_{L^2((-\pi/2,\pi/2) \times \mathbb{R})}  +
 2 \langle \chi f, \Delta_e (\chi f) \rangle_{L^2((-\pi/2,\pi/2) \times \mathbb{R})} +\\+
  \| |\cos{\varphi}|^{-1} \chi f \|^2_{L^2((-\pi/2,\pi/2) \times \mathbb{R})}  \geq
  (\hat C \|(\Delta_e+1) (\chi f)\|_{L^2((-\pi/2,\pi/2) \times \mathbb{R})})^2,
  \end{gather*}
  where $\hat C = \inf \{ |\cos \varphi| \mid \chi(\varphi,t) \not=0 \}$.
  Collecting all terms we obtain
  $$
    \| h \|_{H^{-2}(M)}  \leq \hat C^{-1} \tilde C \sup_{f \in C^\infty_0(\mathbb{R}^2)} \frac{|h(f)|}{\|(\Delta_e+1) f\|_{L^2(\mathbb{R}^2)}}.
  $$
  However,
  \begin{gather*}
    |h_1(f)| = \frac{1}{\sqrt{2 \pi}} \int_{\mathbb{R}^2} \frac{|\hat F_1(\eta)|}{1+\xi^2+\eta^2} (1+\xi^2+\eta^2) \hat f(\xi,\eta) d\xi d\eta \leq \\ \leq
    \frac{1}{2} \| F_1\|_{H^{3/2}(\mathbb{R})} \|(\Delta_e+1) f\| ,\\
     |h_2(f)| = \frac{1}{\sqrt{2 \pi}} \int_{\mathbb{R}^2} \frac{|\hat F_2(\eta) \xi |}{1+\xi^2+\eta^2} (1+\xi^2+\eta^2) \hat f(\xi,\eta) d\xi d\eta\leq  \\ \leq
     \frac{1}{2} \| F_2\|_{H^{1/2}(\mathbb{R})} \|(\Delta_e+1) f\|,
  \end{gather*}
  where $F_1$ and $F_2$ are understood as functions on $\mathbb{R}$ extended by zero from $[0,\ell]$.
 Therefore, we finally obtain
  \begin{gather} \label{sbound1}
    \| h \|_{H^{-2}(M)} \leq C \sqrt{\| F_1\|_{H^{3/2}(\mathbb{R})}^2+\| F_2\|^2_{H^{1/2}(\mathbb{R})}},
  \end{gather}
  where the constant $C$ is given by
  \begin{gather}\label{sbound2}
   C = \frac{1}{\sqrt{2}} (\sup_{(\rho,t) \in \mathrm{supp}{\chi}} \cosh{\rho}) \sqrt{N} 
   \left( 1 + \| \nabla \chi \|_\infty + \| \Delta \chi \|_\infty \right).
  \end{gather}
 If $\Gamma$ consists of several geodesic segments then of course the same inequality holds with the constant being the
 maximum of the constants for the individual segments.
 
 For low lying eigenvalues it is in practice easier to compute the $L^2$-norms of $F_1$ and $F_2$. Of course since the negative
 Sobolev norms are dominated by the $L^2$-norms the above gives estimates of $\| h \|_{H^{-2}(M)}$ in terms
 of the $L^2$-norms of $F_1$ and $F_2$. However, in this case a slightly better constant can be obtained using 
 Theorems \ref{app1} and \ref{app2}. Indeed, equation \ref{sob} means by duality that
 $$
  \| h \|_{H^{-2}(M)} \leq \tilde C \| (\Delta_{\mathbb{H}}+1)^{-1} h \|_{L^2(\mathbb{H})},
 $$
 where $\Delta_{\mathbb{H}}$ is the Laplace operator on the hyperbolic plane.
 By Theorem \ref{app1} and \ref{app2} we have
 \begin{gather} \label{l2errorestimateg2}
  \| (\Delta_{\mathbb{H}}+1)^{-1} h \|_{L^2(\mathbb{H})} \leq \sqrt{C_1^2+C_2^2} \sqrt{\| F_1 \|_{L^2(\mathbb{R})}^2 + \| F_2 \|^2_{L^2(\mathbb{R})}}.
 \end{gather} 
where $C_1$ and $C_2$ are the constants in these theorems.
 
\begin{example}
 For a surface of genus $2$ we use a right angled 12-gon and thus
 have $N=25$.
 Let $L$ be the minimum of the distance between non-adjacent sides.
 For each side $s_i$ we use a cutoff function $\chi_i(\rho,t)=P(\rho/L)$, 
 where $P$ is the function
 $$
  P(x) = \begin{cases} 1, & x<0 \\ 0, & x>1 \\
  1-3 x^2 + 2 x^3, & x \in [0,1],
 \end{cases}
 $$
 and $\rho$ are Fermi-coordinates with respect to the infinite extension of the side
 $s_i$ oriented in such a way that the region $\rho>0$ does not intersect the polygon.
 Note that $\| \nabla \chi_i \| \leq \frac{3}{2 L}$ and $\Delta \chi_i \leq \frac{6}{L^2} + \frac{3}{2 L}$.
 With  $\chi = \chi_1 \cdots \chi_{12}$ we then get
 \begin{gather}
  \| \nabla \chi \|_\infty \leq \frac{3}{L},\\
  \| \Delta \chi \|_\infty \leq \frac{33}{L^2} + \frac{3}{L},
\end{gather}
where we used the well known formula  
$\Delta(\chi_1 \chi_2) = \Delta(\chi_1) \chi_2 + 2 \langle \nabla \chi_1, \nabla \chi_2\rangle + \Delta(\chi_2) \chi_1$
and the fact that near a fixed point all but at most two of the functions $\chi_i$ are constant.
Since $\chi$ is not smooth we can not directly use it in our estimate. However, we can use a suitable regularization
$\chi_\epsilon$ and the fact the the second distributional derivatives of $\chi$ are in $L^\infty$ to show that
the estimate above still holds with the function $\chi$. Collecting all terms we get
\begin{gather*}
\| h \|_{H^{-2}(M)}  \leq 5\sqrt{(C_1^2 + C_2^2)} \left( \frac{33}{L^2} + \frac{6}{L} +1\right)  \sqrt{\| F_1 \|_{L^2(\mathbb{R})}^2 + \| F_2 \|^2_{L^2(\mathbb{R})}}.
\end{gather*}
Since $5\sqrt{(C_1^2 + C_2^2)} \leq 12$ this further simplifies into the easy to use estimate
\begin{gather} \label{l2errorestimateg2}
 \| h \|_{H^{-2}(M)} \leq 12 \left( \frac{33}{L^2} + \frac{6}{L} + 1\right)  \sqrt{\| F_1 \|_{L^2(\mathbb{R})}^2 + \| F_2 \|^2_{L^2(\mathbb{R})}}.
\end{gather}
\end{example}
This estimate improves by a factor of $4.69$ if one is willing to believe the numerical values computed
at the end of Appendix \ref{constest}.

\subsection{Basis functions on hyperbolic cylinders}
\label{subsec:cylinders}

Let $\ell>0$. Then the hyperbolic cylinder $Z_\ell$ is the non-compact manifold obtained 
by factorizing the upper half space by the subgroup
of $SL(2,\mathbb{R})$ generated by the element $\left(\begin{matrix} \mathrm{e}^{\ell/2} & 0 \\ 0 & \mathrm{e}^{-\ell/2}\end{matrix}\right)$.
We will use Fermi coordinates $(\rho,t)$ which are related to the usual coordinates by 
$$
 (x,y) = \mathrm{e}^{t} (\tanh{\rho},\frac{1}{\cosh{\rho}}).
$$
For $L>0$ denote by $Z_\ell^L$ the finite hyperbolic cylinder 
  $$
   Z^L_\ell = \{ x \in Z_\ell \mid -L\leq\rho(x)\leq L \}.
  $$
For $\lambda \in \rz$ let $V^{(\lambda)}(Z^L_\ell)$  be the space of functions $f$ in $C^\infty(Z^L_\ell)$
that satisfy
$$
 (\Delta-\lambda) f =0.
$$
Note that $f \in V^{(\lambda)}(Z^L_\ell)$  has a Fourier expansion of the form
$$
  f(\rho,t)=\sum_{k \in \zz} \left( a_k\;\phi_k(\rho,t) + b_k\; \psi_k(\rho,t) \right),
$$
where 
$(\phi_k)_{k \in \zz}$
satisfy
$$
 (\Delta-\lambda) \phi_k = 0
$$
and have initial data 
\begin{gather*}
 \phi_k(t,0)=\mathrm{e}^{2 \pi \I k t/ \ell},\\
 \partial_\rho \phi_k(t,0)=0,
\end{gather*}
and the functions $\psi_k$
satisfy
$$
 (\Delta-\lambda) \psi_k = 0
$$
and have initial data
\begin{gather*}
 \psi_k(t,0)=0,\\
 \partial_\rho \psi_k(t,0)=\mathrm{e}^{2 \pi \I k t / \ell}.
\end{gather*}
These functions are of the form $\Phi_k(\rho) \mathrm{e}^{2 \pi \I k t/ \ell}$,
where $\Phi_k(\rho)$ solves the ordinary differential equation
\begin{gather} \label{dequ}
 (-\frac{1}{\cosh{\rho}}\frac{d}{d\rho} \cosh{\rho} \frac{d}{d\rho} + \frac{4 \pi^2 k^2}{\ell^2 \cosh^2{\rho}} -\lambda) \Phi_k(\rho)=0
\end{gather}
with the corresponding initial conditions.
A fundamental system of (non-normalized) solutions of this equation, consisting of an even and an odd function, can be given explicitly in terms of hypergeometric functions
\begin{gather*}
 \Phi_k^{even}(\rho)= (\cosh{\rho})^{\frac{2 \pi \I k}{\ell}} \; {}_2 \mathrm{F}_1(\frac{s}{2}+\frac{\pi \I k}{\ell},\frac{1-s}{2}+\frac{\pi \I k}{\ell};\frac{1}{2};-\sinh^2\rho),\\
 \Phi_k^{odd}(\rho)= \sinh{\rho} (\cosh{\rho})^{\frac{2 \pi \I k}{\ell}} \; {}_2 \mathrm{F}_1(\frac{1+s}{2}+\frac{\pi \I k}{\ell},\frac{2-s}{2}+\frac{\pi \I k}{\ell};\frac{3}{2};-\sinh^2\rho),
\end{gather*}
where $\lambda=s(1-s)$ (see \cite{Borthwick:2010fk}, where these functions are analysed). Normalization gives the corresponding solutions to the initial value problems.

Denote by $V_N^{(\lambda)}$ be $4N+2$ dimensional subspace of $V^{(\lambda)}$ spanned by $\phi_k$
and $\psi_k$ with $|k| \leq N$.
Given $f \in V^{(\lambda)}$ we can truncate the Fourier expansion to obtain an element $f^{(N)} \in V_N^{(\lambda)}$ given by
 $$
  f^{(N)}(\rho,t)=\sum_{| k | \leq N } \left( a_k\; \phi_k(\rho,t) + b_k\; \psi_k(\rho,t) \right).
 $$

The $C^1$-norm on the cylinder $Z^L_\ell$ is defined as
$$
  \Vert f \Vert_{C^1}^2 = \Vert f \Vert_{L^\infty}^2 +  \Vert \frac{1}{\cosh{\rho}} \partial_t f \Vert_{L^\infty}^2 +
  \Vert  \partial_\rho f \Vert_{L^\infty}^2.
$$
The following result is crucial for our approximation of eigenfunctions.

\begin{theorem}\label{bestim}
  Suppose  $\phi \in V^{(\lambda)}(Z_\ell)$ is bounded on $Z_\ell$. Let $L>0$ and
  suppose that $N$ is an integer such that $\frac{4 \pi^2 N^2}{\ell^2 \cosh^2{L}} >\lambda$.
  Then,
  \begin{gather*}
   \Vert \phi - \phi^{(N)} \Vert_{L^\infty(Z^L_\ell)} \leq  \tilde \beta_\lambda(N) \| \phi \|_{L^\infty(Z_\ell)},\\
   \Vert \phi - \phi^{(N)} \Vert_{C^1(Z^L_\ell)} \leq \beta_\lambda(N) \| \phi \|_{L^\infty(Z_\ell)},
  \end{gather*}
  where $\tilde \beta_\lambda(N)=A_1(N)$, $\beta_\lambda(N)=\sqrt{A_1(N)^2+A_2(N)^2+A_3(N)^2},$
  \begin{gather*}
  A_1(N) = 4  \sum_{k=N+1}^\infty  \left(\cosh{ c_k(\varphi_0)}\right)^{-1},\\
  A_2(N) = 4  \sum_{k=N+1}^\infty  \frac{2 \pi k}{\ell} \left(\cosh{c_k(\varphi_0)}\right)^{-1},\\
  A_3(N) = 4  \sum_{k=N+1}^\infty  \sqrt{\frac{4 \pi^2 k^2}{\ell^2}-\frac{\lambda}{\cos^2(\varphi_0)}} \left(\sinh{c_k(\varphi_0)}\right)^{-1},
 \end{gather*}
 and
 $$
  c_k(\varphi)=\frac{2 \pi k}{\ell} \arccos\left(\frac{\frac{2 \pi k}{\ell}  \sin
   (\varphi)}{\sqrt{\frac{4 \pi^2 k^2}{\ell^2}-\lambda}}\right)-\sqrt{\lambda}\; \arccot\left(\frac{\sqrt{\lambda} \sin (\varphi)}{\sqrt{\frac{4 \pi^2 k^2}{\ell^2}
   \cos^2(\varphi)- \lambda}}\right).
 $$
\end{theorem}
\begin{remark}
 Both $\beta_\lambda(N)$ and $\tilde \beta_\lambda(N)$ are decreasing exponentially fast in $N$ for fixed $\lambda$
 and give very good bounds for reasonably large $N$ (see Fig. \ref{betas}).
 \begin{figure}[htp] 
\centering
\includegraphics*[width=10cm]{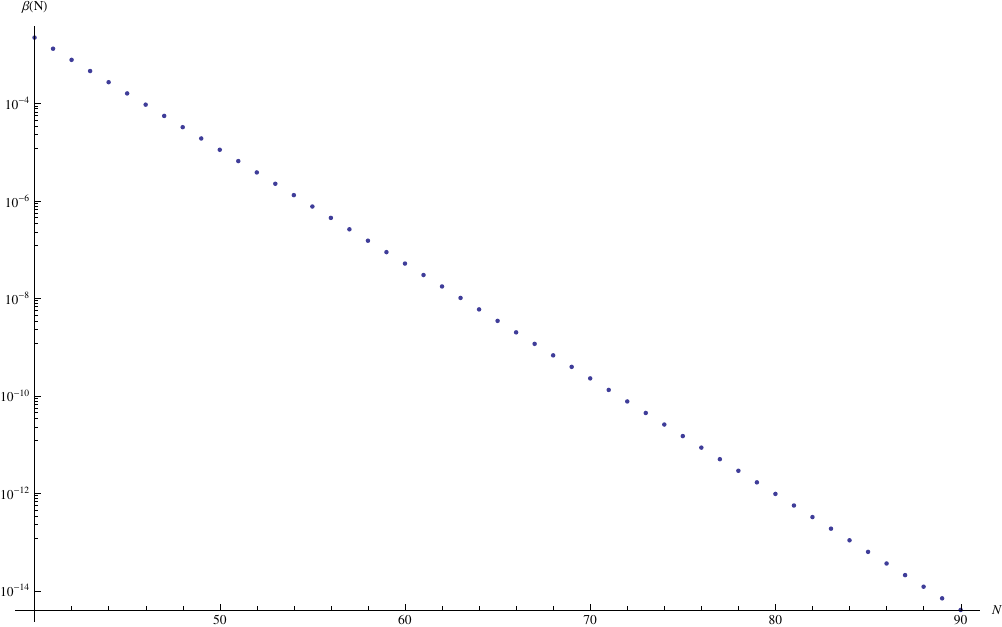}
\caption{$\beta_\lambda(N)$ on a logarithmic scale for $\lambda=30$, $\ell=2\;\mathrm{arccosh}(3+2\sqrt{2})$ and $L = 3/2$. The corresponding cylinder covers a 
fundamental domain for the Bolza surface.}\label{betas}
\end{figure} 
\end{remark}
\begin{proof}[Proof of Theorem \ref{bestim}]
 By rescaling we can assume in the proof that $\| \phi \|_{L^\infty(Z_\ell)} \leq 1$.
 Of course $\psi=\phi - \phi^{(N)}$ has a convergent expansion into
 a Fourier series
 \begin{gather} \label{expansion}
  \psi(\rho,t)=\sum_{| k | > N } \left( a_k\; \phi_k(\rho,t) + b_k\; \psi_k(\rho,t) \right)
 \end{gather}
 which converges uniformly on compact sets.
 Both $\phi_k$ and $\psi_k$ are of the form $\Phi_k(\rho) \mathrm{e}^{2 \pi \I k t/\ell}$
 where $\Phi_k$ solves the ordinary differential equation (\ref{dequ})
 with initial conditions either $\Phi_k(0)=1, \Phi'_k(0)=0$ or $\Phi_k(0)=0, \Phi'_k(0)=1$
 respectively.
 Thus, $\Phi_k$ is either even or odd with respect to the reflection $\rho \mapsto -\rho$.

 Substituting $\varphi=2 \arctan(\tanh{\rho/2})$ we write $\Phi_k(\rho)=\Psi_k(\varphi)$.
 Let
 $$
 V_k(\varphi)=\frac{4 \pi^2 k^2}{\ell^2}-\frac{\lambda}{\cos^2 \varphi}
 $$
 and define $s_k = \arccos\frac{\ell \sqrt{\lambda}}{2 \pi |k|}$ so that
 $V_k$ is non-negative on the interval $[0,s_k]$. Then, by Lemma \ref{growth1},
 we have for all $k>N$ and $\phi<s_k$:
 $$
  \Psi_k(s_k) \geq \Psi_k(\varphi)\cosh{c_k(\varphi)} + \Psi_k'(\varphi)\frac{1}{\sqrt{V_k(\varphi)}} \sinh{c_k(\varphi)}
 $$
 where
 $$
  c_k(\varphi)=\int_{\varphi}^{s_k} \sqrt{V_k(r)} dr.
 $$
 This integral can be computed explicitly and gives the expression in the theorem.
Therefore
 \begin{gather*}
  |\Psi_k(\varphi)|  \leq \left(\cosh{ c_k(\varphi)} \right)^{-1} \Psi_k(s_k) ,\\
  |\Psi_k'(\varphi)| \leq \sqrt{V_k(\varphi)} \left(\sinh{c_k(\varphi)}\right)^{-1}  \Psi_k(s_k).
 \end{gather*}
 Since the $\psi(\rho,t)$ is bounded by $1$ we have $|a_k| |\phi_k(\rho,t)| \leq 1$ and $|b_k| |\psi_k(\rho,t)| \leq 1$.
 For any $\rho>0$. With $\varphi_0=2 \arctan{\tanh{\frac{L}{2}}}$ the above bound then implies that for any 
 $\rho \in [-L,L]$ both $|a_k|  |\phi_k(\rho,t)|$ and $|b_k| |\psi_k(\rho,t)|$
 are smaller equal than
 \begin{gather*}
  \left(\cosh{c_k(\varphi_0)}\right)^{-1}
 \end{gather*}
 and $|a_k| |\frac{\partial}{\partial \rho}\phi_k(\rho,t)|$ and $|b_k| |\frac{\partial}{\partial \rho}\psi_k(\rho,t)|$
 are smaller equal than
 \begin{gather*}
  \sqrt{V_k(\varphi_0)} \left(\sinh{c_k(\varphi_0)}\right)^{-1} .
 \end{gather*}
 Summing the Fourier series one obtains on $Z^L$  the bounds
 $| \psi | \leq A_1$, $| \frac{\partial}{\partial t} \psi | \leq A_2$ and $|\frac{\partial}{\partial \rho} \psi | \leq  A_3$.
 \end{proof}

\subsection{Main estimates for the algorithm for a surface of genus $\genus$} \label{algo}

Suppose that $M$ is a surface of genus $\genus$. Then $M$ can be decomposed into $X$ and $Y$
pieces (pairs of pants). Each of these pieces can be cut open along a shortest geodesic that
connects two different boundary components to obtain a surface that can be identified with a
 closed subset of a hyperbolic cylinder (see Fig. \ref{cyl}).
 \begin{figure}
\centering
\includegraphics*[width=10cm]{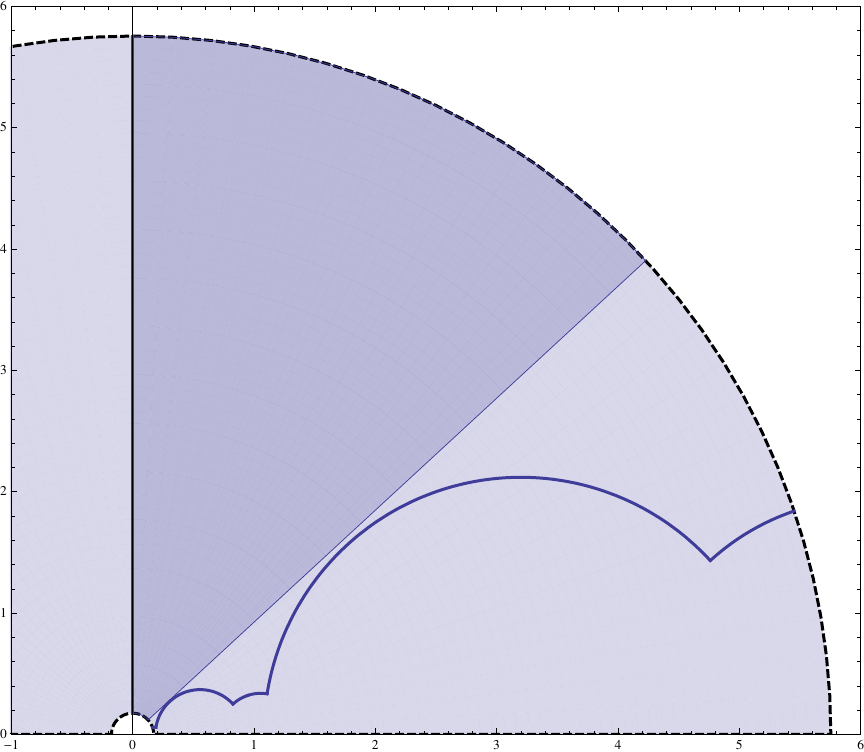}
\caption{A $Y$-piece that was cut open along two geodesic segments covered by a hyperbolic cylinder. The hyperbolic cylinder is obtained from the shaded region by identification of the two dashed circles. The inscribed finite cylinder is double shaded.} \label{cyl}
\end{figure} 
In this way we obtain a collection $(M_i)_{i=1,\ldots,m}$ of closed subsets
of hyperbolic cylinders $Z_{\ell_i}$ with geodesic boundary components in such a way that the defining unique simple
closed geodesic $\gamma_i$ on $Z_{\ell_i}$ is contained in $M_i$. 
Note that we use disjoint cylinders here, i.e.
we use a different copy of $Z_{\ell_i}$ even if the geodesics have the same length and correspond to the
same model in the upper half space.
As $\gamma_i$ is a simple closed geodesic in $M$ the subgroup it generates
in the fundamental group will define a covering map $Z_{\ell_i} \to M$.
Therefore, every eigenfunction $\phi$
on $M$ lifts to a bounded smooth function on $Z_{\ell_i}$ which is an eigenfunction
of the Laplace operator on $Z_{\ell_i}$.

The surface $M$ can of course be constructed from the $M_i$  by glueing, i.e. by identifying
the different boundary geodesics. Suppose that we have a collection
$(\Gamma_\alpha)_{\alpha =1,\ldots,M}$ of geodesic segments and a collection $(\Gamma_{\tilde \alpha})_{\alpha =1,\ldots,M}$ 
so that the geodesic segment $\Gamma_\alpha$ is identified with the segment $\Gamma_{\tilde\alpha}$.
Thus, $\Gamma = \cup_i \Gamma_i$ can be thought of as a finite union of geodesic segments on $M$ and $M \backslash \Gamma$
is the interior of the disjoint union $\coprod_i M_i$ of the $M_i$. This is exactly the setting described in section \ref{mps} where
$\overline{M \backslash \Gamma}$ gets identified with $\coprod_i M_i$ and the segments
$\Gamma_{\alpha}$ and $\Gamma_{\tilde \alpha}$ form the boundary of $\overline{M \backslash \Gamma}$.
Thus, for a function $f \in C^\infty(\coprod_i M_i)$ which is smooth up to the boundary the boundary
values at $\Gamma_{\alpha}$ and $\Gamma_{\tilde \alpha}$ represent the two different boundary values it has
along $\Gamma_\alpha$ if we think of it as a function on $M$.

On each segment $\Gamma_\alpha$ we choose a set $\{x_{\alpha,1},x_{\alpha,2},\ldots, x_{\alpha,n_\alpha}\}$
consisting of an even number
of equidistant points such that $x_{\alpha,1}$ is the origin of the segment and $x_{\alpha,n_\alpha}$
is the endpoint. Let $\{x_{\tilde\alpha,1},x_{\tilde \alpha,2},\ldots, x_{\tilde \alpha,n_{\tilde \alpha}}\}$
be the corresponding points on $\Gamma_{\tilde \alpha}$.
We denote by $\delta_\alpha$ the distance between neighbouring points on the segment $\Gamma_\alpha$
and by $\delta$ the maximum of $(\delta_\alpha)_\alpha$.

Let $N>0$ be a fixed integer and suppose $\lambda>0$. 
Of course on each cylinder $Z_{\ell_i}$ and for each integer $m$ such that $|m| \leq N$
there exists two linearly independent functions on this cylinder which are both eigenfunctions
of the Laplace operator with eigenvalue $\lambda$ and which are of the form
$$
 \Psi(\rho_i,t_i)=\Phi(\rho_i) \mathrm{e}^{\frac{2 \I \pi m t_i}{\ell_i}},
$$
where $(\rho_i,t_i)$ are Fermi coordinates on the corresponding cylinder.
We are choosing the basis functions orthonormal on the largest subset of $M_i$
which is invariant under the rotation symmetry of the cylinder (for a $Y$-piece this
will be one half of a cut-off hyperbolic cylinder and corresponds to the double shaded region
in Fig. \ref{cyl}). Since such a subset will always have
the form $[L_1,L_2] \times S^1$ this can be achieved by choosing
$\Phi(\rho_i)$ orthonormal on this interval.

We extend these functions by zero to the disjoint union of all the cylinders
and think of the collection of all these functions as a basis $(\Psi_k)_{k=1,\ldots,(4N+2)m}$
in the space 
$$
 \oplus_{i=1}^m V^{(\lambda)}_N(Z_{\ell_i}).
$$
We will use this basis to identify vectors in $\mathbb{C}^{(4N+2)m}$ with functions
on the disjoint union of the cylinders, but also with $L^\infty$ functions
on $M$, by restricting the functions to the pieces $M_i$. For $\mathbf{v} \in \mathbb{C}^{(4N+2)m}$
let $\Psi(\mathbf{v})$ be the corresponding function.

For each geodesic segment $\Gamma_\alpha$ we form the matrices
\begin{gather*}
 (\mathbf{A}_{1,\alpha})_{ik} = \sqrt{\frac{\delta_\alpha b_i}{3}} \Psi_k(x_i),\\
 (\mathbf{A}_{2,\alpha})_{ik} = \sqrt{\frac{\delta_\alpha b_i}{3}} \mathbf{n}_\alpha\Psi_k(x_i),\\
 (\mathbf{A}_{3,\alpha})_{ik} = \sqrt{\frac{\delta_\alpha b_i}{3}} \mathbf{t}_\alpha\Psi_k(x_i),
\end{gather*}
where $\mathbf{n}$ is the outward normal unit vector field along the geodesic segment,
$\mathbf{t}$ the normalized tangential vector field, 
and $b_i$ are the coefficients in the composite Simpson rule {\footnote{other coefficients such as those of the Gauss rule may also be used here and lead to analogous error estimates.}}, i.e.
$b_1=1, b_2=4,b_3=2, \ldots, b_{n_\alpha-1}=4,b_{n_\alpha-1}=1$.
Next arrange $\mathbf{A}_{1,\alpha}$ and $\mathbf{A}_{2,\alpha}$ into one matrix $\mathbf{A}_\alpha = \mathbf{A}_{1,\alpha} \oplus \mathbf{A}_{2,\alpha}$, and
add also the tangential derivatives $\mathbf{A}_\alpha^t = \mathbf{A}_{1,\alpha} \oplus \mathbf{A}_{2,\alpha} \oplus \mathbf{A}_{3,\alpha}$.
All the matrices involved map into the same space and we understand the direct sum here  as a direct
sum of linear maps mapping into $\mathbb{C}^{(4N+2)m}$. Thus, the matrix $\mathbf{A}_1 \oplus \mathbf{A}_2$ is obtained from
by attaching all the rows of $\mathbf{A}_2$ as rows to the matrix $\mathbf{A}_1$ (similar to the Matlab notation $[\mathbf{A}_1; \mathbf{A}_2]$).

Similarly, but with one change of sign
\begin{gather*}
 (\mathbf{A}_{1,\tilde \alpha})_{ik} = \sqrt{\frac{\delta_{\tilde \alpha} b_i}{3}} \Psi_k(x_i),\\
 (\mathbf{A}_{2,\tilde \alpha})_{ik} = -\sqrt{\frac{\delta_{\tilde \alpha} b_i}{3}} \mathbf{n}_{\tilde\alpha} \Psi_k(x_i),\\
 (\mathbf{A}_{3,\tilde \alpha})_{ik} = \sqrt{\frac{\delta_{\tilde \alpha} b_i}{3}} \mathbf{t}_{\tilde\alpha} \Psi_k(x_i),
\end{gather*}
and again $\mathbf{A}_{\tilde \alpha} = \mathbf{A}_{1,\tilde \alpha} \oplus \mathbf{A}_{2,\tilde \alpha}$,
$\mathbf{A}_{\tilde \alpha}^t = \mathbf{A}_{1,\tilde \alpha} \oplus \mathbf{A}_{2,\tilde \alpha} \oplus \mathbf{A}_{3,\tilde \alpha}$.
Now define
\begin{gather*}
 \mathbf{A} = \oplus_\alpha \mathbf{A}_{\alpha},\quad \mathbf{A}^t = \oplus_\alpha \mathbf{A}^t_{\alpha}\\
 \mathbf{\tilde A} =  \oplus_\alpha \mathbf{A}_{\tilde \alpha}, \quad  \mathbf{\tilde A}^t =  \oplus_\alpha \mathbf{A}^t_{\tilde \alpha}
\end{gather*}
as well as $\mathbf{B}_{\lambda,N} = \mathbf{A}^t - \mathbf{\tilde A}^t$,
$\mathbf{B}^0_{\lambda,N} = \mathbf{A} - \mathbf{\tilde A}$ and $\mathbf{C}_{\lambda,N} =\mathbf{A} \oplus \mathbf{\tilde A}$. Thus, $\mathbf{B}_{\lambda,N} $ as compared
to $\mathbf{B}^0_{\lambda,N} $ contains also the tangential derivatives.

For each $\mathbf{v} \in \mathbb{C}^{(4N+2)m}$ the vector $\mathbf{B}^0_{\lambda,N} \mathbf{v}$ contains the discretization of the jump of boundary
data of $\Psi(\mathbf{v})$ across $\Gamma$ whereas $\mathbf{C}_{\lambda,N} \mathbf{v}$ contains the discretization of the boundary data itself. 
We will see below that the $L^2$ norm of an eigenfunction on $M$ can be bounded from above and below by the $L^2$-norm of the boundary data and thus
the norm of $\mathbf{C}_{\lambda,N} \mathbf{v}$ will serve as a normalization in our method. To be more precise,
the matrices are chosen such that the norms $\| \mathbf{B}^0_{\lambda,N} \mathbf{v} \|$ and $\| \mathbf{C}_{\lambda,N} \mathbf{v} \|$
are the approximations of $F^\Gamma_{0,0}(\Psi(\mathbf{v}))$
and of the $L^2$ norm of the boundary data of $\Psi(\mathbf{v})$ using the Simpson
rule. The error of composite Simpson integration of a function $g$ on an interval $[a,b]$ with step size $\delta$ 
is bounded from above by
$\frac{\delta^4 |b-a|}{180} \sup_{\xi \in [a,b]} | g^{(4)}(\xi) |$ (see e.g. \cite{MR1423603}, p. 133).
Thus, we have
\begin{gather} \label{simpson01}
  | \| \mathbf{B}^0_{\lambda,N} \mathbf{v} \|^2 -( F^\Gamma_{0,0}(\Psi(\mathbf{v})))^2 | \leq \\ \nonumber
  \leq \frac{\delta^4 \ell_\Gamma}{180} \left(\| D_t^4(\Psi(\mathbf{v})^2) |_\Gamma\|_{L^{\infty}(\Gamma)} +
  \| D_t^4 \mathbf{n}_\Gamma(\Psi(\mathbf{v})^2) |_\Gamma\|_{L^{\infty}(\Gamma)}\right),\\
  \label{simpson02}
  | \| \mathbf{C}_{\lambda,N} \mathbf{v} \|^2 - \left( \| (\Psi(v)|_\Gamma \|_{L^2(\Gamma)}^2  + \| \mathbf{n}  \Psi(\mathbf{v})) |_\Gamma \|_{L^2(\Gamma)}^2  \right)|\leq \\ \leq
  \frac{\delta^4 \ell_\Gamma}{180}  \left(\| D_t^4(\Psi(\mathbf{v})^2) \|_{L^\infty(\Gamma)} 
  + \|  D_t^4 \mathbf{n}_\Gamma(\Psi(\mathbf{v})^2) |_\Gamma\|_{L^{\infty}(\Gamma)}\right), \nonumber
\end{gather}
where $D_t$ denotes the tangential derivative.

Moreover, as we chose the basis functions orthonormal with respect to the $L^2$-inner product of a
subset of $M$, we have
\begin{gather} \label{geq01}
 \| \Psi(\mathbf{v}) \|_{L^2(M)} \geq \| \mathbf{v} \|.
\end{gather}

Assume now that $\lambda$ is an eigenvalue of the Laplace operator and suppose that $\phi$
is an $L^2$-normalized eigenfunction. We may truncate the Fourier expansion of this eigenfunction
on each of the cylinders to obtain a function $\phi^{(N)}$ which then corresponds to a fixed vector
$\mathbf{v}$ such that $\phi^{(N)} = \Psi(\mathbf{v})$.
We choose $N$ large enough such that Theorem \ref{bestim} applies and from the fact that
$\phi$ is bounded by an explicitly computable constant (see Appendix, Corollary \ref{lbound})
the following estimates are explicit
\begin{gather}
 \| (\phi^{(N)} - \phi)|_\Gamma\|_{L^\infty(\Gamma)} \leq \beta_\lambda(N) \| \phi \|_{\infty},\\
 \| (\mathbf{n}\phi^{(N)} - \mathbf{n}\phi) |_\Gamma \|_{L^\infty(\Gamma)} \leq  \beta_\lambda(N) \| \phi \|_{\infty}.
\end{gather}
The error of the composite Simpson integration 
of a function $g$ on an interval $[a,b]$ with step size $\delta$ is also bounded by
$\delta |b-a| \sup_{\xi \in [a,b]} | g'(\xi) |$ as one can easily see by replacing $g$ in each sub-interval $[x_k,x_{k+1},x_{k+2}]$
of length $2 \delta$ by the constant function $g(x_{k+1})$.
Therefore\footnote{We use this estimate instead of (\ref{simpson02}) because a 
bound on the first derivative of eigenfunctions is easier to obtain explicitly and we do not need good accuracy for this bound}, we obtain
\begin{gather*}
 | \| \mathbf{C}_{\lambda,N} \mathbf{v} \|^2 - \left( \| \phi |_\Gamma \|_{L^2(\Gamma)}^2 + \|  \mathbf{n} \phi  |_\Gamma\|_{L^2(\Gamma)}^2 \right) | 
 \leq  2 \beta_\lambda(N)^2 \|\phi\|_{\infty}^2 + \\ +
 (2 \delta \ell_\Gamma) \left(\| D_t (\phi^2) |_\Gamma\|_{L^{\infty}(\Gamma)} + 
\| D_t \mathbf{n}_\Gamma (\phi^2) |_\Gamma\|_{L^{\infty}(\Gamma)}\right)
\end{gather*}
From the bound (\ref{l2bound}) of the Appendix and the fact that $\phi$ is $L^2$-normalized we conclude that
$$
   \left( \| \phi |_\Gamma \|_{L^2(\Gamma)}^2 + \|  \mathbf{n} \phi  |_\Gamma\|_{L^2(\Gamma)}^2 \right)^{\frac{1}{2}} > c_1(\lambda),
$$
where $c_1$ is explicitly computable and depends only on $\lambda$ and the geometry of $M$ and $\Gamma$.
Using the bound $L^\infty$ and the $C^1$ bound in the Appendix of eigenfunctions we can compute
constants $N_c$ and $\delta_c$ such that for all $N>N_c$ and $\delta < \delta_c$ 
we have
$$
 \| \mathbf{C}_{\lambda,N} \mathbf{v} \|  >  \frac{c_1(\lambda)}{2}.
$$
Note that this is a very crude estimate and we have used only the linear error estimate in Simpson's
rule. This was done so that we can use the explicit estimates that we proved in the appendix.
The critical constants $\delta_c$ and $N_c$
guarantee that the numerical integration using the Simpson rule together with the approximation
of $\phi$ by the truncated function together yield a relative error of not more than 50\%. For the method it will be enough
to show that $\| \mathbf{C}_{\lambda,N} \mathbf{v} \|$ is bounded from above so that this estimate is sufficient.

Let $\sigma_1(\mathbf{B}^0_{\lambda,N})$ be the smallest
singular value of $\mathbf{B}^0_{\lambda,N}$ and
$\sigma_1(\mathbf{B}_{\lambda,N},\mathbf{C}_{\lambda,N})$ the smallest generalized singular value (see e.g. \cite{MR1417720} or also \cite{MR2337576} in the context of MPS) of the pair
$(\mathbf{B}_{\lambda,N},\mathbf{C}_{\lambda,N})$. Here the generalized singular values of $\mathbf{B}_{\lambda,N} \in \mathrm{Mat}(m_2 \times m_1)$ with respect to 
$\mathbf{C}_{\lambda,N} \in \mathrm{Mat}(m_3 \times m_1)$ are defined
as the singular values of $\mathbf{B}_{\lambda,N}$ as an operator from $\mathbb{C}^{m_1}$ to $\mathbb{C}^{m_2}$ where the Euclidean
norm $\| \cdot \|$ on $\mathbb{C}^{m_1}$ is replaced by the norm $\| \cdot \|_{\mathbf{C}_{\lambda,N}}$ defined by 
$\| \mathbf{v} \|_{\mathbf{C}_{\lambda,N}}:=\| \mathbf{C}_{\lambda,N} \mathbf{v}\|$.
Note that
$$
 \sigma_1(\mathbf{B}_{\lambda,N},\mathbf{C}_{\lambda,N}) = \inf_{\mathbf{v} \not= 0} \frac{\|\mathbf{B}_{\lambda,N} \mathbf{v} \|}{\| \mathbf{C}_{\lambda,N} \mathbf{v}\|}.
$$

If $N$ is large enough, so that
the assumptions of Theorem \ref{bestim} are satisfied, we have
$$
  \sigma_1(\mathbf{B}_{\lambda,N},\mathbf{C}_{\lambda,N}) \leq \frac{2 \sqrt{2}}{c_1(\lambda)} \beta_\lambda(N) \| \phi \|_\infty
 $$
in case $\lambda$ is an eigenvalue. This follows immediately from the fact that an exact eigenfunction
satisfies the boundary conditions.

Let as usual $\lambda=s(1-s)$ with $\mathrm{Im}(s) \geq 0$ and define $\sigma=\mathrm{Re} (s)$ and $r = \mathrm{Im}(s)$.
Assume that $\lambda'$ is not an eigenvalue and let $s',\sigma'$ and $r'$ defined as above.
For each $i$ we choose a hyperbolic cylinder $Z^L$ with $L$ large enough
such that $M_i \subset Z^L$ and such that $M_i$ is relatively compact in the interior of $Z^L$, so that
the distance between $\partial Z^L$ and $M_i$ greater than a positive number $d>0$. For the truncated eigenfunction
$\phi^{(N)}$ consider the boundary data $\phi^{(N)}_Z \oplus \mathbf{n} \phi^{(N)}_Z$ on $\partial Z^L$.
Using the resolvent kernel of the Laplace operator \footnote{this is also often referred to as the Green's function of the Helmholz operator}
on the hyperbolic cylinder (see Appendix \ref{resestcly}) we can reconstruct $\phi^{(N)}$ from its boundary data
$$
 \phi^{(N)}(x) = \int_{\partial Z^L} k^{Z}_s(x,x') \mathbf{n}_{x'} \phi^{(N)}_Z(x')  -  \phi^{(N)}_Z(x') \mathbf{n}_{x'}  k^{Z}_s(x,x') dx'.
$$
Now define
$$
  \phi^{(N)}_{s'}(x) = \int_{\partial Z^L} k^{Z}_{s'}(x,x') \mathbf{n}_{x'} \phi^{(N)}_Z(x')  -  \phi^{(N)}_Z(x') \mathbf{n}_{x'}  k^{Z}_{s'}(x,x') dx'.
$$
Then
\begin{gather*}
 \phi^{(N)}(x) -  \phi^{(N)}_{s'}(x) =\\= \int_{s'}^s
 \int_{\partial Z^L} \frac{\partial}{\partial \tilde s} k^{Z}_{\tilde s}(x,x') \mathbf{n}_{x'} \phi^{(N)}_Z(x')  -  \phi^{(N)}_Z(x') \mathbf{n}_{x'} \frac{\partial}{\partial \tilde s} k^{Z}_{\tilde s}(x,x') dx' d\tilde s.
\end{gather*}
Using Lemma \ref{greenestcylder} we obtain
$$
 \| \phi^{(N)} -  \phi^{(N)}_{s'} \|_{C^1(M)} \leq |s-s'| \tilde C_{M,s,s'}  \left( \| \phi^{(N)}_Z \|_{L^2(\partial Z)}^2 + \| \mathbf{n} \phi^{(N)}_Z \|_{L^2(\partial Z)}^2 \right)^{\frac{1}{2}}.
$$
Of course the $L^2$-norm of the boundary data is bounded from above by the $L^2$-norm of the boundary
data of the exact eigenfunction $\phi$ so that we obtain the bound
$$
 \| \phi^{(N)} -  \phi^{(N)}_{s'} \|_{C^1(M)} \leq |s-s'| \tilde C_{M,s,s'}  \sqrt{\mathrm{Vol}(\Gamma)} \| \phi \|_{C^1(M)}.
$$

\begin{theorem} \label{maindist}
 With the notation above we have for all $N>N_c$ and $\delta < \delta_c$
 $$
  \sigma_1(\mathbf{B}_{\lambda',N},\mathbf{C}_{\lambda',N}) \leq \frac{2 \sqrt{2}}{c_1(\lambda)} \left( \beta_\lambda(N) \| \phi \|_\infty +  |s-s'| \tilde C_{M,s,s'}   \sqrt{\mathrm{Vol}(\Gamma)} \| \phi \|_{C^1}  \right)
 $$
 if $\lambda$ is an eigenvalue with normalized eigenfunction $\phi$. Here
 $\| \phi \|_\infty$ and $\| \phi \|_{C^1}$ are bounded by explicit constants (Corollary \ref{lbound} and \ref{lboundzwei}).
\end{theorem}

Moreover, by Theorem \ref{main01} and the inequalities (\ref{simpson01}) and (\ref{geq01}) we have:

\begin{theorem} \label{mainm01}
 Suppose that $\mathbf{v} \in \mathbb{C}^{(4N+2)m}$ is a unit vector and $\epsilon=\|\mathbf{B}^0_{\lambda,N} \mathbf{v} \|$ and let
 $$\tau=C(\epsilon+ \frac{\delta^4 \ell_\Gamma}{180}  \left(\| D_t^4(\Psi(\mathbf{v})^2) \|_{L^\infty(\Gamma)} + 
 \|  D_t^4 \mathbf{n}_\Gamma(\Psi(\mathbf{v})^2) |_\Gamma\|_{L^{\infty}(\Gamma)}\right)) < 1,$$
 where $C=\tilde C \sqrt{C_1^2+C_2^2}$ are the constants in Section  \ref{s:4.1}. Then, there exists an eigenvalue in 
 the interval $[\lambda-\frac{(1+\lambda) \tau}{1-\tau},\lambda+\frac{(1+\lambda) \tau}{1-\tau}]$.
 In particular the estimate holds if $\epsilon=\sigma_1(\mathbf{B}^0_{\lambda,N})$ and $\mathbf{v}$ is a corresponding normalized singular vector.
\end{theorem}

\begin{remark}
 The term $\epsilon+ \frac{\delta^4 \ell_\Gamma}{180}  \left(\| D_t^4(\Psi(\mathbf{v})^2) \|_{L^\infty(\Gamma)}  + 
 \|  D_t^4 \mathbf{n}_\Gamma(\Psi(\mathbf{v})^2) |_\Gamma\|_{L^{\infty}(\Gamma)}\right)$
 can be minimized as another singular value problem, using the bounds on the derivatives of the basis functions
 that follow from the differential equation together with the Lemmata \ref{oni1} and \ref{oni3}.
 This results in a matrix, whose singular value is precisely the error estimate.
 Actual computations show that in practice the major contribution of the error comes from the first term.
 \end{remark}
 
 \begin{remark}
  An analog of Theorem  \ref{maindist} holds for the $n$-th relative singular values $\mathbf{B}_{\lambda',N}$
  with respect to $\mathbf{C}_{\lambda',N}$. In this case the term $|s-s'|$ gets replaced by
  the maximum of $\{ |s_k-s'| \mid k=1\ldots,n \}$, where $s_k$ correspond to $n$ pairwise orthogonal
  eigenvalues. We do not state this here as we want to keep the exposition simple.
  In the same way Theorem \ref{mainm01} applies to situations with multiplicities and eigenvalues that are close
  to each other (see remark \ref{remrem}).
 \end{remark}

\subsection{Description of the algorithm for a surface of genus $\genus$} \label{algozwei}

Since the fundamental solutions of the differential equation (\ref{dequ}) satisfy the bounds of Lemmata
\ref{oni1} and \ref{oni3} their higher derivatives can also be bounded explicitly simply by using the differential
equation.
We use coordinates $r=\sinh{\rho}$ and $t$ on the hyperbolic cylinders, so that the differential equation
is of the form
\begin{gather} \label{dequ2}
 \left(-(1+r^2)\frac{d}{dr} (1+r^2) \frac{d}{dr}  + \frac{4 \pi^2 k^2}{\ell^2} - (1+r^2)\lambda \right) \Phi_k(r) =0.
\end{gather}
Using subdivision into smaller intervals and the Taylor expansion around the boundary points of these intervals yields an
approximation $\Phi_k^{(M)}(r)$ to the exact solution by piecewise defined polynomial splines. Using the bounds
on the higher derivatives the error in the $C^1$-norm can be explicitly bounded from above. 
Since the measure on the hyperbolic
cylinder is equal to $dr dt$ in these coordinates, normalization of the basis functions can be done
within this error, simply by integrating these piecewise polynomials.
In this way, the matrices $\mathbf{B}^0_{\lambda,N}$, $\mathbf{B}_{\lambda,N}$ and $\mathbf{C}_{\lambda,N}$ can be computed with explicit error bounds.
The generalized singular value decomposition can then be done using any method from linear algebra.
Once a numerical generalized singular value decomposition has been obtained the error of the singular values
can be estimated from above as follows. Suppose that $\mathbf{U} \mathbf{D} \mathbf{V}$ is a numerically obtained
generalized singular value decomposition of $\mathbf{B}$ with respect to $\mathbf{C}$. Thus,  $\mathbf{D}$ is diagonal and $\mathbf{U}$ is unitary within a 
certain precision.
Moreover, $\mathbf{V}$ is unitary with respect to the inner product induced by $\mathbf{C}$, again within a certain precision.
Thus, $\| \mathbf{U}^* \mathbf{U} - 1 \| \leq \delta_1$ and  $\| \mathbf{V}^* \mathbf{C}^* \mathbf{C} \mathbf{V}  - \mathbf{C}^* \mathbf{C} \| \leq \delta_2$,
and $\| \mathbf{B} - \mathbf{U} \mathbf{D} \mathbf{V} \| \leq \delta_3$, where these three numbers reflect the errors. All three numbers can be
estimated from above using interval arithmetics for example by bounding the operator norm by the Hilbert-Schmidt
norm. Since the matrices $\mathbf{U}_0=\mathbf{U} (\mathbf{U}^* \mathbf{U})^{-1/2}$ and 
$\mathbf{V}_0=\mathbf{V} (\mathbf{V}^* \mathbf{C}^* \mathbf{C} \mathbf{V} )^{-1/2} (\mathbf{C}^* \mathbf{C})^{1/2}$
are unitary in the respective inner products the generalized singular values of $\mathbf{U}_0 \mathbf{D} \mathbf{V}_0$
are exactly the diagonal entries of $\mathbf{D}$. By the characterization of generalized singular values of $\mathbf{B}$ as the norm
distance of $\mathbf{B}$ to the rank $k$ operators it follows that the error of the generalized singular values of $\mathbf{B}$ bounded by
$$
 \| (\mathbf{B} - \mathbf{U}_0 \mathbf{D} \mathbf{V}_0) (\mathbf{C}^* \mathbf{C})^{-1/2}\|.
$$
Using functional calculus a longer calculation shows that this is bounded by
$$
 \frac{\delta_3}{\sigma_1(\mathbf{C})} + \| \mathbf{D} \| \left( \frac{\delta_1}{2} + \frac{\delta_2+\delta_1 \delta_2}{\sigma_1(\mathbf{C})^3} \right),
$$
where $\sigma_1(\mathbf{C})$ is the smallest singular value of $\mathbf{C}$. Here we assumed that $\delta_2 \leq \sigma_1(\mathbf{C})$.
Note that the smallest singular value of $\mathbf{C}$ can also be estimated from below using exactly the same
procedure for the ordinary singular value decomposition of $\mathbf{C}$. Thus, interval arithmetics can be applied to
obtain a lower bound for the generalized singular values of $\mathbf{B}$ with respect to $\mathbf{C}$.

An interval  $I$ is then tested for eigenvalues by computing a lower bound for the value of
$\sigma_1(\mathbf{B}_{x,N},\mathbf{C}_{x,N})$ for a discrete set of points $x$ in $I$. 
By theorem \ref{maindist} for large enough $N$
the distance between the points can be chosen such that eigenvalues can occur only near points where
the lower bound is small enough. For a fixed $\epsilon>0$ this algorithm then will normally find a discrete set
of points $y_i$ such that eigenvalue can occur only in $\epsilon$-neighborhoods of these points.
The presence of eigenvalues and their multiplicities can then be tested using the singular values
of $\mathbf{B}^0_{\lambda,N}$. An upper bound for small singular values can easily be obtained by first
finding the singular value decomposition numerically and then using the numerically obtained
singular vectors $\mathbf{v}$ to compute $\| \mathbf{B}^0_{\lambda,N} \mathbf{v} \|$. Since the derivatives of the basis functions
can be bounded using the differential equation (\ref{dequ}) and the bounds of Lemmata
\ref{oni1} and \ref{oni3} the error terms in Theorem \ref{mainm01} can directly be estimated from above.
Using interval arithmetics one can therefore obtain rigorous interval inclusions for the eigenvalues.

\begin{remark}
 Once the presence of a single eigenvalue is established in a small interval theorem
 \ref{maindist} can also be used to establish an error bound by narrowing the interval in which
 the eigenvalue can be located.
\end{remark}

The analysis of rigorous error bounds can be simplified considerably 
by replacing the exact basis functions with the piecewise defined polynomial basis functions 
themselves. These are of course orthogonal on any cylinder, thus normalization involves only integration
of piecewise defined polynomials which can be done explicitly.
Thus, the quasi-mode becomes a finite linear combination of products of piecewise
defined polynomials in $r$ and Fourier modes in $t$.  The so constructed function 
does not satisfy the equation $(\Delta-\lambda) \Phi=0$
but instead $(\Delta-\lambda) \Phi=\chi$, where due to the nature of the differential equation
$(1+r^2) \chi$ is again constructed out of piecewise defined polynomials and Fourier modes. Its $L^2$-norm
can very easily be estimated from above by integrating over the smallest cylinders that contain $M_i$.
The fourth derivatives in the error of Simpson's rule in theorem \ref{mainm01} can also 
be estimated more directly in this case as a bound can easily obtained from the coefficients
of the polynomial.

\section{Implementation and examples} \label{examples}

We have implemented our method in three different programs. 
\subsection{Non-rigorous implementations}
The first two programs are Fortran
programs. One works solely for genus $2$ surfaces and takes as coordinates
$mw$-Fenchel-Nielsen coordinates. The second program
takes as an input a labeled graph (encoding the way the surface is glued from $Y$-pieces)
and the Fenchel-Nielsen coordinates.  Both programs then follow the algorithm
described in this article.
The  basis functions are computed for each value of $\lambda$
by numerically solving the differential equation (\ref{dequ}) using the DLSODA routine from ODEPACK (see \cite{odepack01,odepack02}).
Normalization of the basis functions is achieved by numerical integration
of the solution of (\ref{dequ}).
In order to compute the singular values
$\sigma_1(\mathbf{B}^0_{\lambda,N})$ we use LAPACK routines. 
The generalized singular values $\sigma_1(\mathbf{B}_{\lambda,N},\mathbf{C}_{\lambda,N})$
are computed using $\mathbf{QR}$-decomposition.
A search algorithm then looks for the small minima of $\sigma_1(\mathbf{B}_{\lambda,N},\mathbf{C}_{\lambda,N})$ as a function
of $\lambda$.
A small enough step size guarantees that no eigenvalues are missed in the process.
The other singular values are also necessary in the search: a small second singular value
of $\mathbf{B}^0_{\lambda,N}$ implies that either the eigenvalue has multiplicity greater than one, or another
eigenvalue is very close. Our method can not distinguish between multiplicities
and close eigenvalues if they can not be separated within the given precision.
Both programs achieve accuracies close to machine precision for the low lying eigenvalues when the surface
is far enough from the boundary of Teichm\"uller space (so that the DLSODA algorithm can compute the basis functions
accurately enough). The code for the genus two program is available to the scientific community under GPLv3 
and can be obtained from \url{http://www1.maths.leeds.ac.uk/~pmtast/hyperbolic-surfaces/hypermodes.html}.

Another program was written in Mathematica. It uses Taylor's method on intervals of size of order $1/\sqrt{\lambda}$
in order to solve
the differential equation (\ref{dequ}) and compute the matrices $\mathbf{B}_{\lambda,N}$ and $\mathbf{C}_{\lambda,N}$
with arbitrarily high precision. The linear algebra such as $\mathbf{QR}$-decomposition and
search for the minimum is done within Mathematica. The achieved accuracy depends on the
computing time and the machine precision used. We were able to compute the first 70
digits of the first eigenvalue for the Bolza surface (see Section \ref{bolzasection}) using this program.

\subsection{Rigorous numerical implementations}
To demonstrate the power of our method
we modified the genus two Fortran and Mathematica codes 
to give rigorous error estimates for the eigenvalues.
As we described in the previous section we use the Taylor series method to solve 
the differential equation (\ref{dequ}) within a given accuracy and then follow the analysis
described in  section \ref{algozwei}.
The corresponding eigenvalue inclusions can be considered rigorous if one carefully
includes rounding errors in the analysis and is willing to trust the implementation of 
basic mathematical functions in Fortran and Mathematica.

\subsection{The Bolza surface} \label{bolzasection}

In order to demonstrate our method we consider a special surface of genus $2$.
The so-called Bolza surface has Fenchel-Nielsen $mw$-coordinates
\begin{gather*}
(\ell_1,t_1;\ell_2,t_2;\ell_3,t_3)=\\=
  (2 \,\arccosh{(3+2 \sqrt{2})},\frac{1}{2};2\, \arccosh{(1+\sqrt{2})},0;2\, \arccosh{(1+\sqrt{2})},0).
\end{gather*}
This surface is known to maximize the length of the systole in genus $2$ and it also maximizes
the order of the symmetry group (see e.g. \cite{MR749104,MR1250756,MR1305753}).
The first eigenvalue on the Bolza surface was estimated by Jenni \cite{MR749104} remarkably 
accurately to be in the interval $[3.83,3.85]$. Aurich and Steiner \cite{Aurich:1989} (see also \cite{ASS88})give the value 
$3.8388$ which was obtained using the finite element method.
Figure \ref{sigbolza} shows a plot of the singular value after $\mathbf{QR}$-decomposition ($N=60$, $\delta=0.001$) of the 
matrix $A \oplus B$ and projection onto the first direct summand.
Note that at the minima the function is very small and order $10^{-12}$.
\begin{figure}[htp]
\centering \label{sigbolza}
\includegraphics*[width=11cm]{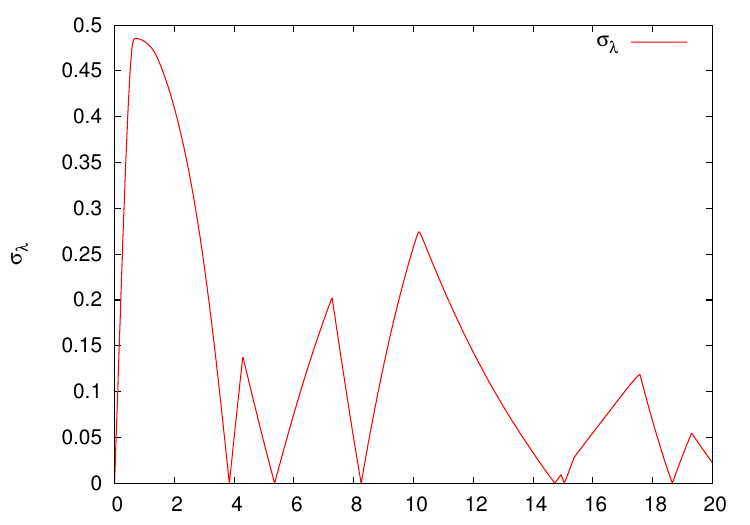}
\caption{Smallest singular value after $\mathbf{QR}$-decomposition for the Bolza surface as a function of $\lambda$.}
\end{figure} 

The search algorithm in the Fortran program finds the first eigenvalue within $8$ digits of precision.
Using multiple precision in Mathematica we obtained the value.
$$
 \lambda_1 \approx 3.8388872588421995185866224504354645970819150157
$$
where we believe that all digits are correct. The rigorous implementation that uses very simple bounds in its
current form gives a rigorous error bound of $10^{-6}$ (which can be improved by simply investing more computing time).

We believe that $\lambda_1$ is important as we conjecture that this is the maximal
value of the first eigenvalue of the Laplace operator for constant negative curvature $-1$
in genus $2$. This conjecture is supported by our calculations in Teichm\"uller space (the corresponding
analysis will be discussed elsewhere) and is in line with the findings in  (\cite{MR2202582}),
where it is shown \footnote{this proof relies on a lower bound for the eigenvalue of a mixed Dirichlet-Neumann problem that
was obtained numerically from a finite element method} that the maximal value of the first eigenvalue
of the Laplace operator with respect to any metric fixing the volume is attained for a singular
metric on the Bolza surface (\cite{MR2202582}).

\section{Calculation of the Spectral Zeta function and the Spectral Determinant} \label{zeta}

The spectral zeta function $\zeta_{\Delta}(s)$ is defined as the meromorphic continuation of the function
$$
 \zeta_{\Delta}(s)=\sum_{i=1}^\infty \lambda_i^{-s},
$$
where $(\lambda_i)_{i \in \nz_0}$ are the eigenvalues repeated according to their multiplicities.
The above sum converges for $\Re(s)>1$. As usual a meromorphic continuation to the whole complex plane
is constructed using the Mellin transform
$$
 \zeta_{\Delta}(s) = \frac{1}{\mathrm{\Gamma(s)}} \int_0^\infty t^{s-1} \left( \mathrm{tr}(\mathrm{e}^{-\Delta t})-1 \right) dt.
$$
and the fact that the heat kernel has an asymptotic expansion, by splitting the above integral
into an integral over $[0,1]$ and one over $[1,\infty)$. From the heat expansion it follows that 
zero is not a pole, which allows one to define the zeta-regularized determinant 
${\det}_\zeta(\Delta)$ of the Laplace operator by
$$
 \log {\det}_\zeta(\Delta) = - \zeta'_{\Delta}(0).
$$
In the examples we will also consider the value of the zeta function at the point $-1/2$.
In physics this value is often referred to as the Casimir energy or vacuum energy.

For our purposes we use the following splitting
$$
 \zeta_{\Delta}(s) = \frac{1}{\mathrm{\Gamma}(s)} \left(\int_0^\epsilon t^{s-1} \left( \mathrm{tr}(\mathrm{e}^{-\Delta t})-1 \right) dt + 
 \int_\epsilon^\infty t^{s-1} \left( \mathrm{tr}(\mathrm{e}^{-\Delta t})-1 \right) dt \right),
$$
where $\epsilon>0$.

The Selberg trace formula applied to the heat trace (see e.g. Equ. (193) in \cite{Mkl04}) is
\begin{gather} \label{stf}
 \mathrm{tr}(\mathrm{e}^{-\Delta t}) = \frac{\mathrm{Vol}(M)\mathrm{e}^{-\frac{t}{4}}}{4 \pi t} \int_0^\infty \frac{\pi \mathrm{e}^{-r^2 t}}{\cosh^2(\pi r)} dr +
 \sum_{n=1}^\infty \sum_{\gamma}  \frac{\mathrm{e}^{-t/4}}{\sqrt{4 \pi t}} 
  \frac{\ell(\gamma) \mathrm{e}^{-\frac{n^2 \ell(\gamma)^2}{4 t}}}{2 \sinh\frac{n \ell(\gamma)}{2}},
\end{gather}
where the second sum in the second term is over the set of primitive closed geodesics $\gamma$.
Splitting the integral and using the standard analytic continuation gives for $\Re(s)> -N$
$$
  \zeta_{\Delta}(s) =  \frac{1}{\mathrm{\Gamma}(s)}( T_1^{\epsilon}(s)+ T_2^{\epsilon,N}(s) + 
  T_3^{\epsilon,N}(s) + T_4^{\epsilon,N}(s)),
$$
where
\begin{gather*}
  T_1^\epsilon(s) =  \sum_{i=1}^\infty \lambda_i^{-s} \mathrm{\Gamma}(s,\epsilon \lambda_i),\\
  T_2^{\epsilon,N}(s) = \sum_{k=0}^N \frac{a_k \epsilon^{s+k-1}}{s+k-1},\\
  T_3^{\epsilon,N}(s) = \frac{\mathrm{Vol}(M)}{4 \pi} \int_0^\infty  I_N(r) dr,\\ 
  T_4^{\epsilon,N}(s) = \sum_{n=1}^\infty \sum_{\gamma} \int_0^\epsilon t^{s-1} \frac{\mathrm{e}^{-t/4}}{\sqrt{4 \pi t}} 
  \frac{\ell(\gamma) \mathrm{e}^{-\frac{n^2 \ell(\gamma)^2}{4 t}}}{2 \sinh\frac{n \ell(\gamma)}{2}} dt,
\end{gather*}
where
$$
  I_N^\epsilon(r) = \int_0^\epsilon t^{s-2}  \left( \mathrm{e}^{-(r^2+\frac{1}{4}) t } - 
 \sum_{k=0}^N \frac{(-1)^k}{k!} (r^2+\frac{1}{4})^k t^k \right) dt,
$$
the $a_k$ are given by
$$
 a_k = \frac{\mathrm{Vol}(M)}{4 \pi} \int_0^\infty \frac{(-1)^k}{k!} \frac{\pi (r^2+1/4)^k}{\cosh^2(\pi r)} dr - \delta_{1,k},
$$
and $\mathrm{\Gamma}(x,y)$ denotes the incomplete Gamma function
$$
 \mathrm{\Gamma}(x,y) = \int_y^\infty t^{x-1} \mathrm{e}^{-t} dt.
$$

Differentiation results in the following formula for the spectral determinant.

$$
 -\log {\det}_{\zeta} \Delta = \zeta'_\Delta(0) = L_1 + L _2 + L_3
$$
where
\begin{gather*}
 L_1^\epsilon = \sum_{i=1}^\infty \mathrm{\Gamma}(0,\epsilon \lambda_i),\\
 L_2^\epsilon = - \frac{\mathrm{Vol}(M)}{4 \pi \epsilon} -\left(\frac{\mathrm{Vol}(M)}{12 \pi} + 1\right)(\gamma + \log(\epsilon)) 
  +  \frac{\mathrm{Vol}(M)}{4} \times \\ \int_0^\infty 
 \text{sech}^2(\pi  r) \left( \frac{1-\mathrm{E}_2 \left(\epsilon (r^2+\frac{1}{4})\right)}{\epsilon}+(r^2+\frac{1}{4})
  \left(\gamma-1+\log(\epsilon(r^2+1/4)) \right) \right) dr,\\
  L_3^\epsilon =  \sum_{n=1}^\infty \sum_{\gamma} \int_0^\epsilon \mathrm{e}^{-t/4} \frac{\ell_i \mathrm{e}^{-\frac{n^2 \ell(\gamma)^2}{4 t}}}{4 \sqrt{\pi } t^{3/2} \sinh{\left(\frac{1}{2} n \ell(\gamma) \right)}} dt,
\end{gather*}
and $\mathrm{E}_2(x)$ is the generalized exponential integral which equals $x \;\mathrm{\Gamma}(-1,x)$.
All the integrals have analytic integrands and can be truncated with exponentially small error. They can therefore
be evaluated to high accuracy using numerical integration.

For fixed $s$ and $\epsilon>0$ not too small the sums over the eigenvalues converge rather quickly  and therefore 
$T_1^\epsilon(s)$ and $L_1^\epsilon$ can be computed quite accurately from the first eigenvalues only.
The error made by summing only over  finitely many eigenvalues can be estimated directly by Lemma \ref{heattrb}.
For not too large $\epsilon>0$ the sums over the length spectrum are also very fast convergent.
As they are exponentially decaying as $\epsilon$ tends to zero one can either neglect this contribution
by choosing $\epsilon$ small enough or one can compute them 
using a finite part of the length spectrum. The error introduced
in this way is also exponentially decreasing as $\epsilon \to 0$ and it can be explicitly estimated using 
 the Selberg trace formula  as follows. If the length of the shortest closed geodesic is $L$ the second term in
 the Selberg trace formula Equ. (\ref{stf}) is for all $t < T<\sqrt{L^2+1}-1$ bounded by
\begin{gather}\label{errheattr}
 F_T(t) = \sqrt{\frac{T}{t}} \tr(\mathrm{e}^{-\Delta T}) \mathrm{e}^{\frac{T}{4}+\frac{L^2}{4T}} \mathrm{e}^{-\frac{L^2}{4t}}
\end{gather} 
and the heat trace for a fixed $T$ can be estimated by Lemma \ref{heattrb}. Therefore,
we have for all $T> \epsilon$
\begin{gather*}
 | T_4^{\epsilon,N}(s) | \leq \int_0^\epsilon t^{|s-1|} F_T(t) dt\\
 | L_3^\epsilon | \leq \int_0^\epsilon \frac{1}{2 t} F_T(t) dt,
\end{gather*}
and the right
hand side converges exponentially fast to zero as $\epsilon \to 0$.
 An analogous formula holds for the error if elements in the length
spectrum up to length $L$ are taken into account.

Of course the short geodesics dominate the contribution from the length spectrum. 
Note however that as a consequence of the collar theorem
there are at most $3 \genus-3$ closed geodesics of length smaller that $2 \,\mathrm{arcsinh}(1)	 \approx 1.7627$
on a surface of genus $\genus>1$ (see e.g. \cite{Buser:1992}). 

\section{Spectral Determinants and Casimir Energy for special surfaces} \label{zetaexa}

In this section we compute the values of the Casimir energy and the spectral determinant for
the three isolated surfaces of genus two with large symmetry group as well as for other examples
of genus two and genus three surfaces that have been treated in the literature before. All Fenchel-Nielsen coordinates
are given in $mw$-form, i.e. the $3$-valent graph describing this decomposition consists of two
vertices that are connected by three edges. For all the surfaces we first used the Fortran double precision 
implementation of our program to generate a list of eigenvalues. Even though our algorithm allows us to choose the
stepsize in such a way that we are guaranteed not to miss an eigenvalue the current implementation
does not achieve that automatically. In our computations we have chosen to check completeness of the list of eigenvalues
by using the Selberg trace formula applied to the heat trace. This is possible as the absolute value of the contribution of the length spectrum in Selberg's trace formula
is bounded from above by  $F_T(t)$, defined in Equ. (\ref{errheattr}), so that missed eigenvalues can easily be detected comparing the heat trace computed
from the Selberg trace formula. Note that unlike heuristic checks using Weyl's law this gives a rigorous way to check completeness up to a certain value.
The data files containing the list of first non-zero eigenvalues for all the examples below have been uploaded to the arXiv
together with this article. They can also be found at the link \url{http://www1.maths.leeds.ac.uk/~pmtast/publications/eigdata/datafile.html}. 

\subsection{The Bolza surface}

The Bolza surface has symmetry group $S_4 \times \mathbb{Z}_2$. Its order, $48$, thus maximizes
the order of the symmetry group (of orientation preserving isometries) amongst all genus $2$ hyperbolic surfaces. The Bolza surface is referred to
as the regular octagon (and sometimes as the Hadamard-Gutzwiller model) in the physics literature
as it can be obtained from a regular octagon in hyperbolic space by identifying opposite sides.
Its Fenchel-Nielsen coordinates are
\begin{gather*}
(\ell_1,t_1;\ell_2,t_2;\ell_3,t_3)=\\=
  (2\,\arccosh{(3+2 \sqrt{2})},\frac{1}{2};2\,\arccosh{(1+\sqrt{2})},0;2\, \arccosh{(1+\sqrt{2})},0).
\end{gather*}

We computed the spectral determinant as well as the Casimir energy $\zeta(-1/2)$ and obtained
\begin{gather*}
 \mathrm{det}_\zeta(\Delta) \approx 4.72273280444557,\\
 \zeta_\Delta(-1/2) \approx -0.65000636917383,
\end{gather*}
where we believe that all decimal places are correct. This accuracy was obtained
by computing the first $10$ eigenvalues with a $30$ digit precision, the first
$500$ eigenvalues in quad precision and subsequently
by using the simple length spectrum up to $\ell=9$. We relied on the multiplicities of the length spectrum
as  computed in \cite{MR980200}. 
Note that the spectral determinant and
the Casimir energy can both be computed accurately within 5 decimal places without taking
the length spectrum into account at all, using the eigenvalues only.
Completeness of the list of the first $500$ eigenvalues was checked using a list of $700$ computed
eigenvalues. To compute the first $500$ eigenvalues of the Bolza surface to a precision of $12$ digits about $10000$
$\lambda$-evaluations of generalized singular value decomposition were needed. This took about
$10$ minutes on a $2.5$ GHz Intel Core i5 quad core processor (where parallelization was used).

\begin{figure}[htp]
\centering \label{zetaplot}
\includegraphics*[width=11cm]{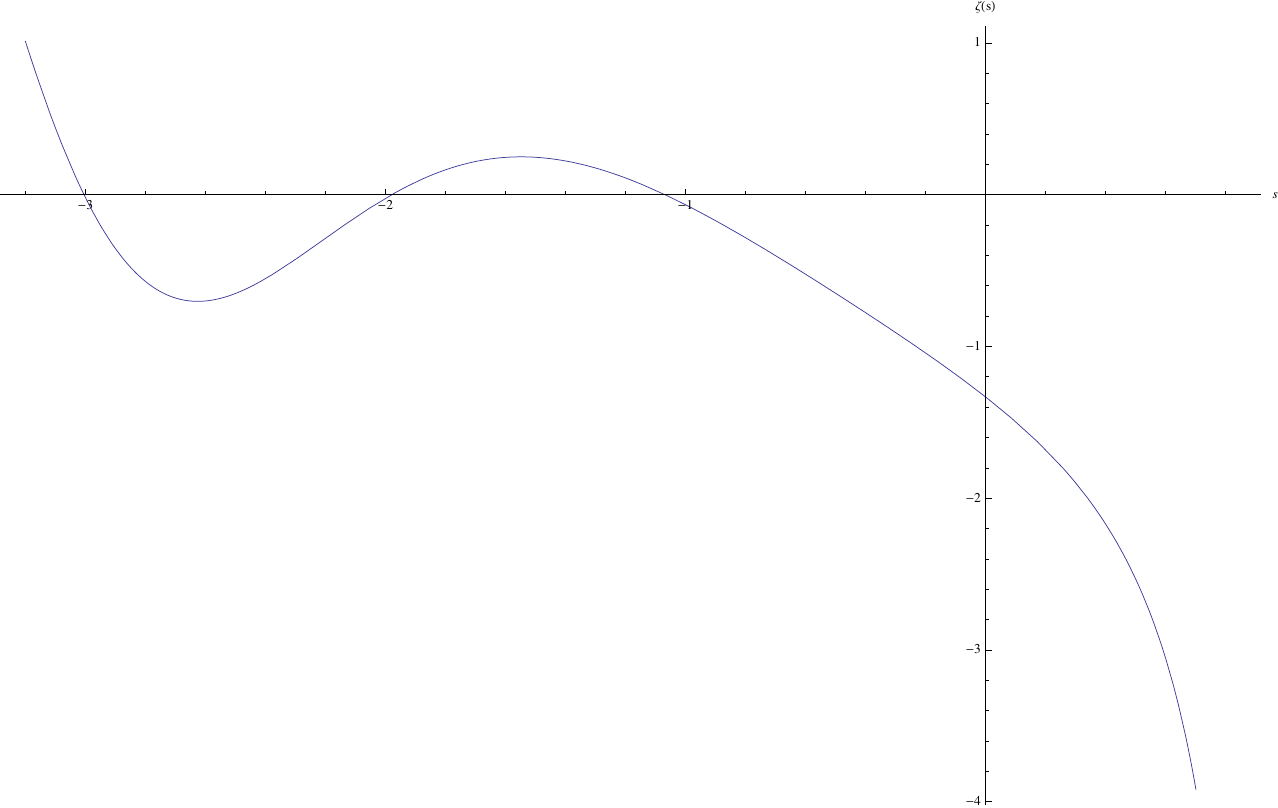}
\caption{$\zeta_\Delta(s)$ as a function of $s$ for the Bolza surface}
\end{figure} 

We conjecture that the spectral determinant is maximized in genus $2$ for the Bolza surface and we therefore
believe that the above number is important. Our computations show that Bolza surface as a point in Teichm\"uller space
is indeed a local maximum. Note that for the Bolza surface is known to be a critical point for the spectral
determinant.

\subsection{The surface with symmetry group $D_6 \times \mathbb{Z}_2$} \label{bolz}

This surface has Fenchel-Nielsen coordinates (see \cite{Buser:2005}, sec. 3.5)
$$
(\ell_1,t_1;\ell_2,t_2;\ell_3,t_3)=(2\, \arccosh(2),0;2 \,\arccosh(2),0;2 \,\arccosh(2),0).
$$
It is also known to be a critical point for any modular invariant such as the Casimir energy or the spectral
determinant.
Computing the first $500$ eigenvalues with double precision and neglecting the contribution from
the length spectrum other than from the $12$ primitive closed geodesics of length
$2 \, \arccosh(2)$ we obtain
\begin{gather*}
 \mathrm{det}_\zeta(\Delta) \approx 4.428000668,\\
 \zeta_\Delta(-1/2) \approx -0.67250924.
\end{gather*}

\subsection{The surface with symmetry group $\mathbb{Z}_5 \times \mathbb{Z}_2$}

This surface has Fenchel-Nielsen coordinates (see \cite{Buser:2005}, sec. 3.5)
\begin{gather*}
(\ell_1,t_1;\ell_2,t_2;\ell_3,t_3)= \\=(2 \,\arccosh{\frac{3+\sqrt{5}}{2}},0; 2\,\arccosh{\frac{2+\sqrt{5}}{2}},\frac{1}{2}; 
2\,\arccosh{\frac{9+3 \sqrt{5}}{4}},\frac{1}{2}).
\end{gather*}
It is also known to be a critical point for any modular invariant such as the Casimir energy or the spectral
determinant.
Computing the first $500$ eigenvalues with double precision and neglecting the contribution from
the length spectrum we obtain
\begin{gather*}
 \mathrm{det}_\zeta(\Delta) \approx 4.630575,\\
 \zeta_\Delta(-1/2) \approx -0.656986.
\end{gather*}

\subsection{The Gutzwiller octagon} \label{regoct}

This is a surface obtained from a regular hyperbolic octagon in the upper half space
using a side identification different from that for the Bolza surface (see \cite{Ninnemann:1995} where also the low
lying spectrum of this surface was computed).
The Fenchel-Nielsen coordinates of this surface are given by
\begin{gather*}
  (\ell_1,t_1;\ell_2,t_2;\ell_3,t_3)=\\=
  (4 \,\arccosh{\frac{\sqrt{2}+1}{\sqrt{2}}},\frac{1}{4};
   2 \,\arccosh{\frac{\sqrt{2}+1}{\sqrt{2}}},\frac{1}{2};2 \,\arccosh{\frac{\sqrt{2}+1}{\sqrt{2}}},\frac{1}{2}).
\end{gather*}

Computing the first $500$ eigenvalues with double precision and neglecting the contribution from
the length spectrum we obtain
\begin{gather*}
 \mathrm{det}_\zeta(\Delta) \approx 3.76048,\\
 \zeta_\Delta(-1/2) \approx -0.72747.
\end{gather*}

\subsection{The example of Aurich and Steiner}

Eigenvalue statistics of a certain generic octagon was investigated in \cite{MR1214552}. The Fenchel-Nielsen
parameters of this surface are roughly (see \cite{ABCKS}) \footnote{see Acknowledgements}.
\begin{gather*}
(\ell_1,t_1) \approx (3.717414183638, 0.0304758025243),\\
(\ell_2,t_2) \approx (3.303402988815, 0.2711895405183),\\
(\ell_3,t_3) \approx (3.408324727953, 0.2187479424048).
\end{gather*}
Computing the first $500$ eigenvalues with double precision and neglecting the contribution from
the length spectrum we obtain
\begin{gather*}
 \mathrm{det}_\zeta(\Delta) \approx 3.959168,\\
 \zeta_\Delta(-1/2) \approx -0.715195.
\end{gather*}

\subsection{Examples of Rocha and Pollicott}

These are the examples treated in \cite{MR1474163}, where approximate values for their
spectral determinants are given.

The first example has Fenchel-Nielsen coordinates 
$$
(\ell_1,t_1;\ell_2,t_2;\ell_3,t_3)= (5.05,0;1,0;0.9,0).
$$
Computing the first $500$ eigenvalues with double precision and taking
into acount only the short primitive geodesics in the length spectrum we obtain
\begin{gather*}
 \mathrm{det}_\zeta(\Delta) \approx 0.395833,\\
 \zeta_\Delta(-1/2) \approx -1.817507.
\end{gather*}

The second example has Fenchel-Nielsen coordinates 
$$
(\ell_1,t_1;\ell_2,t_2;\ell_3,t_3)= (0.98,0;3.5,0;0.98,0).
$$
Computing the first $500$ eigenvalues with double precision and taking
into acount only the short primitive geodesics in the length spectrum we obtain
\begin{gather*}
 \mathrm{det}_\zeta(\Delta) \approx 0.6114618,\\
 \zeta_\Delta(-1/2) \approx -1.6541313.
\end{gather*}

The third example has Fenchel-Nielsen coordinates 
$$
(\ell_1,t_1;\ell_2,t_2;\ell_3,t_3)= (5,0;1,0;1,0).
$$
Computing the first $500$ eigenvalues with double precision and taking
into acount only the short primitive geodesics in the length spectrum we obtain
\begin{gather*}
 \mathrm{det}_\zeta(\Delta) \approx 0.5124672,\\
 \zeta_\Delta(-1/2) \approx -1.6591527.
\end{gather*}

The values do not agree with the values obtained in  \cite{MR1474163}.
We believe that not enough lengths of primitive geodesics were taken
into account there.

\subsection{An example of genus three}

The following example has genus three and is a double cover of the Bolza surface.
It is obtained by cutting the two copies of the Bolza surface open along the two geodesics
of length $2\,\arccosh{(1+\sqrt{2})}$ used in its pants decomposition (see section \ref{bolz})
The two resulting surfaces of type $(0,4)$ are then glued by identifying each boundary
with the the corresponding boundary on the other copy of the surface.
We computed the first $600$ eigenvalues on this surface and obtained
\begin{gather*}
 \mathrm{det}_\zeta(\Delta) \approx 9.6507,\\
 \zeta_\Delta(-1/2) \approx -0.7802.
\end{gather*}
For this accuracy the contribution of the length spectrum can be neglected.
\appendix

\section{Estimates on the basis functions}

This section deals with growth estimates and bounds on the function
$\Phi(\rho)$ that satisfies the differential equation
\begin{gather}
 (-\frac{1}{\cosh{\rho}}\frac{d}{d\rho} \cosh{\rho} \frac{d}{d\rho} + \frac{A}{\cosh^2{\rho}} -\lambda) \Phi(\rho)=0
\end{gather}
with initial conditions
$$
 \Phi(0) =a, \quad \frac{d}{d \rho} \Phi(0) = b.
$$
Substituting $\varphi=2 \arctan(\tanh{\rho/2})$ this can also be written
as
$$
 (-\frac{d^2}{d\varphi ^2} + A - \frac{\lambda}{\cos^2{\varphi}}) \Psi(\varphi)=0,
$$
where $\Psi(\varphi)=\Phi(\rho(\varphi))$ satisfies $\Psi(0) =a$ and  $\frac{d}{d \rho} \Psi(0) = b$.
The derivatives in different coordinates are related by $\cosh{\rho} \frac{d}{d\rho} = \frac{d}{d\varphi}$.
Now define a  potential
$$
 V(\varphi)= A - \frac{\lambda}{\cos^2{\varphi}}.
$$

\begin{lemma} \label{oni3}
 Suppose that $\lambda>0$ and $V(\varphi)$ does not vanish on an interval $I \subset [0,\pi/2)$.
 Then the function $E(\varphi)=|\Psi(\varphi)|^2 - \frac{1}{V(\varphi)} |\Psi'(\varphi)|^2$ is non-increasing on $I$.
\end{lemma}
\begin{proof}
 Simple differentiation gives
 $$
  E'(\varphi) = - \frac{1}{(V(\varphi))^2}  \frac{2 \lambda \sin{\varphi}}{\cos^3{\varphi}} |\Psi'(\varphi)|^2 \leq 0,
 $$
 which implies the statement.
\end{proof}

\begin{lemma} \label{growth1}
 Suppose that $\lambda>0$ and $V(\varphi) \geq 0$ on some interval $[\varphi_0,\varphi_1] \subset [0,\pi/2)$
 and assume $\Phi(\varphi_0) \geq 0$ and $\Phi'(\varphi_0) \geq 0$.
 Define
 $$
  c_{\varphi_0}(\varphi)= \int_{\varphi_0}^\varphi \sqrt{V(t)} dt.
  $$
 Then for all $\varphi \in [\varphi_0,\varphi_1]$ we have
 $$
  \Phi(\varphi) \geq \Phi(\varphi_0) \cosh \left (c_{\varphi_0}(\varphi) \right) + \frac{1}{\sqrt{V(\varphi_0)}} \Phi'(\varphi_0) \sinh\left( c_{\varphi_0}(\varphi) \right).
 $$
 and
 $$
 \Phi'(\varphi) \geq \sqrt{V(\varphi)} \Phi(\varphi_0) \sinh \left (c_{\varphi_0}(\varphi) \right) +  \sqrt{\frac{V(\varphi)}{V(\varphi_0)}} \Phi'(\varphi_0) \cosh \left( c_{\varphi_0}(\varphi) \right).
 $$
\end{lemma}
\begin{proof}
 Since every solution can be obtained as a sum of one with $\Phi(\varphi_0) = 0$ and one with $\Phi'(\varphi_0) = 0$
 we can assume without loss of generality that either $\Phi(\varphi_0) = 0$ or $\Phi'(\varphi_0) = 0$. We also
 assume that $\Phi(\varphi)$ is non-zero.
 Define 
 $$
  h(\varphi):=\Phi(\varphi_0) \cosh \left (c_{\varphi_0}(\varphi) \right) + \frac{1}{\sqrt{V(\varphi_0)}} \Phi'(\varphi_0) \sinh\left( c_{\varphi_0}(\varphi) \right).
 $$
 Then $h(\varphi)$ and $\Phi(\varphi)$ have the same initial data at $\varphi_0$. It follows from the differential
 equation and the assumption $\Phi(\varphi_0) \geq 0$ and $\Phi'(\varphi_0) \geq 0$ that 
 $\Phi(\varphi), \Phi'(\varphi)$ and $\Phi''(\varphi)$ are positive on some interval $(\varphi_0,\delta)$. 
 Moreover, where $\Phi(\varphi)>0$ and $\Phi'(\varphi)>0$ the function $\Phi(\varphi)$
 is convex and therefore,  $\Phi(\varphi), \Phi'(\varphi)$ and $\Phi''(\varphi)$ are non-decreasing functions
 on $[\varphi_0,\varphi_1]$.
 In particular $\Phi(\varphi) > 0$ on $(\varphi_0,\varphi_1]$. The differential equation  implies
 $$
  (\Phi'(\varphi) h(\varphi) - h'(\varphi) \Phi(\varphi))' = (V(\varphi)- \frac{h''(\varphi)}{h(\varphi)}) \Phi(\varphi) h(\varphi).
 $$
 The right hand side of this is non-negative as on can easily see by direct calculation that
 $$
  (V(\varphi)- \frac{h''(\varphi)}{h(\varphi)}) = \lambda \frac{\tan(\varphi)}{\sqrt{V(\varphi)}\cos^2(\varphi)} T((c_{\varphi_0}(\varphi))\geq 0,
 $$
 where $T(\varphi)$ is either $\tanh(\varphi)$ or $\coth(\varphi)$ depending on the inital conditions.
 Thus,
 $$
  (\ln{\frac{\Phi(\varphi)}{h(\varphi)}})'=(\Phi'(\varphi) h(\varphi) - h'(\varphi) \Phi(\varphi)) \geq 0
 $$
 on $(\varphi_0,\varphi_1]$ and consequently
 $\Phi(\varphi) \geq h(\varphi)$. The equation also implies that $\Phi(\varphi)-h(\varphi)$ is non-decreasing which gives the inequality for the derivatives.
\end{proof}

The following lemma is a direct consequence of Gronwall's inequality applied to the first order system
$$
 \frac{d}{d\varphi} \left( \begin{matrix} \Phi(\varphi) \\ \frac{1}{c}\Phi'(\varphi) \end{matrix} \right) = 
 \left( \begin{matrix} 0 & c \\ V(\varphi)/c & 0 \end{matrix} \right) 
 \left( \begin{matrix} \Phi(\varphi) \\ \frac{1}{c}\Phi'(\varphi) \end{matrix} \right).
$$

\begin{lemma} \label{oni2}
 Suppose that $\lambda>0$, $c>0$ and $|V(\varphi)| \leq c^2$ on some interval $[\varphi_0,\varphi_1] \subset [0,\pi/2)$
 Then for all $\varphi \in [\varphi_0,\varphi_1]$ we have
 $$
  (|\Phi(\varphi)|^2 + \frac{1}{c^2}  |\Phi'(\varphi)|^2)^{\frac{1}{2}}  \leq 
  \mathrm{e}^{c |\varphi-\varphi_0|}(|\Phi(\varphi_0)|^2 + \frac{1}{c^2}  |\Phi'(\varphi_0)|^2)^{\frac{1}{2}}.
 $$
\end{lemma}

For intervals $(\varphi_0,\varphi_1)$ where $V(\varphi)$ is positive this can be slightly refined. 
\begin{lemma} \label{oni1}
 Suppose that $\lambda>0$ and $V(\varphi) >0$ on some interval $(\varphi_0,\varphi_1) \subset [0,\pi/2)$
 Then for all $\varphi \in (\varphi_0,\varphi_1)$ we have
 $$
  (|\Phi(\varphi)|^2 + \frac{1}{V(\varphi)}  |\Phi'(\varphi)|^2)^{\frac{1}{2}}  \leq 
  \frac{V(\varphi_0)}{V(\varphi)} \mathrm{e}^{c_{\varphi_0}(\varphi)}(|\Phi(\varphi_0)|^2 + \frac{1}{V(\varphi_0)}  |\Phi'(\varphi_0)|^2)^{\frac{1}{2}}.
 $$
\end{lemma}
\begin{proof}
 The equation is equivalent to the system
$$
 \frac{d}{d\varphi} \left( \begin{matrix} \Phi(\varphi) \\ \frac{1}{\sqrt{V(\varphi)}}\Phi'(\varphi) \end{matrix} \right) = 
 \left( \begin{matrix} 0 &  \sqrt{V(\varphi)}\\ \sqrt{V(\varphi)} & -\frac{V'(\varphi)}{2V(\varphi)} \end{matrix} \right) 
 \left( \begin{matrix} \Phi(\varphi) \\ \frac{1}{\sqrt{V(\varphi)}}\Phi'(\varphi) \end{matrix} \right).
$$
The operator norm of the family of matrices in this system at the point $\varphi$ is given by
$$
 -\frac{V'(\varphi)}{2V(\varphi)} + \sqrt{(\frac{V'(\varphi)}{2V(\varphi)})^2+V(\varphi)}
$$
which is bounded from below by
$$
 |\frac{V'(\varphi)}{V(\varphi)}| + \sqrt{V(\varphi)}
$$
given that $V'(\varphi) \leq 0$.
The lemma now follows immediately from Gronwall's inequality.
\end{proof}

\section{Estimates for the resolvent on the upper half space} \label{constest}

Let $\Delta_\mathbb{H}$ be the Laplace operator on the upper half space equipped with the Poincare
metric. Then the integral kernel $k_s(z,w)$ of the resolvent
$$
 R_{s(1-s)}(\Delta_\mathbb{H})=(\Delta_\mathbb{H} -s(1-s))^{-1}
$$
is given by
$$
 k_s(z,w) = Q_{-s}(\cosh\mathrm{dist}(z,w)).
$$
Here $Q_s$ is the Legendre Q-function. 
Using the point pair invariant $u(z,w)=\frac{\cosh{\mathrm{dist}(z,w)}-1}{2}=\frac{|z-w|^2}{4 \Im(z) \Im(w)}$
this can be expressed as 
$$
 k_s(z,w)=\frac{1}{2 \pi}Q_{-s}(1+2u(z,w))=F_s(u(z,w)),
$$
where $F_s(u)$ is given by an integral (see e.g. \cite{MR1942691})
$$
 F_s(u):=\frac{1}{2 \pi} Q_{-s}(\cosh{\rho}) = \frac{1}{4 \pi} \int_0^1 (\xi(1-\xi))^{s-1} (\xi+u)^{-s} d\xi.
$$ 
This formula can be used to construct a meromorphic
continuation of the resolvent as a function of $s$ to the whole complex plane.
We will need the following estimates on this function all of which follow right from the above integral
formula.
\begin{lemma} \label{lem1}
 Suppose that $\sigma=\mathrm{Re}(s)>0$. Then 
 \begin{gather*}
  |F_{s}(u)| \leq \frac{1}{4 \pi} \frac{(\mathrm{\Gamma}(\sigma))^2}{\mathrm{\Gamma}(2\sigma)} u^{-\sigma},\\
  |\frac{d}{du}F_{s}(u)| \leq \frac{|s|}{4 \pi} \frac{(\mathrm{\Gamma}(\sigma))^2}{\mathrm{\Gamma}(2\sigma)} u^{-\sigma-1}.
 \end{gather*}
\end{lemma}

\begin{lemma} \label{lem2}
 Suppose that $\sigma=\mathrm{Re}(s)>1$. Then 
 \begin{gather*}
  |F_{s}(u)| \leq \frac{1}{4 \pi} \log{\frac{1+u}{u}},\\
   |\frac{d}{du}F_{s}(u)| \leq \frac{|s|}{4 \pi} \frac{1}{u(1+u)}.
 \end{gather*}
\end{lemma}

A similar estimate as in the above lemma also holds for $\mathrm{Re}(s)>0$
(c.f. \cite{MR1942691}, Lemma 1.7, where however (1.48) seems to be inaccurate)

\begin{lemma} \label{lem3}
 Suppose that $\sigma=\mathrm{Re}(s)>0$. Then, there are constants $C_1(\sigma),C_3(\sigma)$, 
 depending only on $\sigma$, and a constant $C_2(s)$, depending only on $s$, such that
 \begin{gather*}
  | F_{s}(u) |  \leq \frac{1}{4 \pi} \log{\frac{1+u}{u}}  +  C_1(\sigma),\\
  | \frac{d}{du}F_{s}(u) | \leq \frac{C_2(s)}{4 \pi} \frac{1}{u} + C_3(\sigma).
 \end{gather*}
 Moreover, $C_2(s)$ can be bounded by $C_4(\sigma) |s|$, where $C_4(\sigma)$ depends
 only on $\sigma$.
\end{lemma}

In the important case $\sigma=1/2$ one can choose $C_1(1/2)=1$, $C_2(1/2)=2$ and $C_3(1/2)=1$,
although these choices are not optimal. The statement that $C_2(s)$ can be bounded by
$C_4(\sigma) |s|$ is not optimal either for large values of 
$\Im(s)$, where cancellations determine the asymptotic behaviour as $u \to 0$. 

\begin{theorem} \label{app1}
Let $\Gamma$ be a geodesic segment in the upper half space and let 
 $(\rho,t)$ be Fermi coordinates with respect to the infinite extension of $\Gamma$.
 Suppose a distribution $h_1$ is defined as
 $$
  h_1(f) = \int F_1(t) f(0,t) dt,
 $$
 where $F_1 \in L^2(\mathbb{R})$. 
 Then 
 $$
 \|(\Delta_\mathbb{H}+3/2)^{-1} h_1 \|_{L^2(\mathbb{H})} \leq C_1 \| F_1 \|_{L^2(\mathbb{R})}
 $$
 where the constant $C_1$ is given by (\ref{const1}) and satisfies $C_1 \leq 1.75$.
\end{theorem}
\begin{proof}
 Since the Laplace operator commutes with the $SL(2,\mathbb{R})$-action
 we can assume without loss of generality that the geodesic segment is part
 of the imaginary axis.
 The function $(\Delta_\mathbb{H}+3/2)^{-1} h_1$ can be computed by convolving $h_1$
 with the integral kernel $k_s(z,w)$ for $s=3/2$. We use Fermi coordinates
 $(\rho,t)$ so that the integral kernel is given by $k_s(\rho,t;\rho',t')=k_s(\rho,t-t';\rho',0)$.
 Thus,
 $$
   ((\Delta_\mathbb{H}+3/2)^{-1} h_1 )(\rho,t)= \int F_1(t') k_s(\rho,t-t',0,0) dt'.
 $$
 Since this is a convolution the generalized Young inequality applies so that
 $$
  \int |((\Delta_\mathbb{H}+3/2)^{-1} h_1 )(\rho,t)|^2 dt \leq  \|F_1\|_{L^2}^2 \left( \int | k_s(\rho,t,0,0) | dt \right)^2.
 $$
 The value of the point pair invariant $u(z,w)$  at $w=\I$ in these coordinates is given by
 $$
  u(\rho,t,0,0)=\frac{1}{2}\left( \cosh{t} \cosh{\rho} -1\right),
 $$
 so that the inequality follows with
 \begin{gather}\label{const1}
  C_1= \left(\int (\int | F_s(\frac{1}{2}\left( \cosh{t} \cosh{\rho} -1\right)) | dt)^2 \cosh{\rho} d\rho \right)^{\frac{1}{2}}.
 \end{gather}
 The integral over $t$ can be estimated for $\cosh{\rho}>3$ as
 \begin{gather*}
  \int | F_s(\frac{1}{2}\left( \cosh{t} \cosh{\rho} -1\right)) | dt \leq \frac{1}{4 \pi} \frac{(\mathrm{\Gamma}(3/2))^2}{\mathrm{\Gamma}(3)}
  (2)^{3/2} (\cosh{\rho})^{-3/2} \int_0^\infty (\mathrm{e}^{|t|}/2 -1/3)^{-\frac{3}{2}} dt =\\
  =  \frac{3}{8} \sqrt{\frac{3}{2}} \left(4 \sqrt{2}-\pi -2 \arcsin{\sqrt{\frac{2}{3}}}+2 \arctan{\frac{1}{\sqrt{2}}}\right) (\cosh{\rho})^{-3/2} 
 \end{gather*}
 For $\cosh{\rho}<3$ we split the integral over $t$ into two parts. One with $\cosh{t}>3$ and one with
 $\cosh{t}<3$. The former gives
 \begin{gather*}
  \int_{\cosh{t}>3} | F_s(\frac{1}{2}\left( \cosh{t} \cosh{\rho} -1\right)) | dt \leq \\ \leq \frac{1}{4 \pi} \frac{(\mathrm{\Gamma}(3/2))^2}{\mathrm{\Gamma}(3)}
  2^{3/2}  \int_{\cosh{t}>3}  (\cosh{t} -1)^{-\frac{3}{2}} dt =\\=
  \frac{1}{16} \left(\sqrt{2}+\log\tanh \left(\frac{1}{4} \log (3+2\sqrt{2} )\right)\right)
 \end{gather*}
 and the latter
 \begin{gather*}
  \int_{\cosh{t}<3} | F_s(\frac{1}{2}\left( \cosh{t} \cosh{\rho} -1\right)) | dt \leq \frac{1}{4 \pi} \int_{\cosh{t}<3} \log{\frac{\cosh{t}+1}{\cosh{t}-1}} dt 
  \leq \\ \leq\frac{1}{2 \pi} \int_0^{\mathrm{arccosh}{3}} \log{\frac{t^2+4}{t^2}} dt =
  \frac{\arccosh(3) \log \left(1+\frac{4}{\arccosh(3)^2}\right)+4 \arctan \left(\frac{1}{2} \arccosh(3)\right)}{2 \pi }
 \end{gather*}
 Squaring and integrating finally gives
 \begin{gather*}
  C_1 \leq \frac{1}{8}\sqrt{ \left(3-2 \sqrt{2}\right)  \alpha_1^2 + \sqrt{2} (\alpha_2)^2} \approx 1.7485475
 \end{gather*}
 where
 \begin{gather*}
 \alpha_1= 3 \pi +6\; \arccot\left(2
   \sqrt{2} \right) -12 \sqrt{2},\\
   \alpha_2=\sqrt{2}+\log \left(\tanh \left(\frac{1}{4} \log \left(3+2
   \sqrt{2}\right)\right)\right)+\\+\frac{8 \left(\arccosh(3) \log
   \left(1+\frac{4}{\arccosh(3)^2}\right)+4 \arctan\left(\frac{1}{2} \arccosh(3)\right)\right)}{\pi }.
  \end{gather*}
\end{proof}

\begin{theorem} \label{app2}
 Let $\Gamma$ be a geodesic segment in the upper half space suppose
 $(\rho,t)$ are Fermi coordinates with respect to the infinite extension of $\Gamma$.
 Suppose a distribution $h_2$ is defined as
 $$
  h_2(f) = -\int_0^\ell F_2(t) (\frac{\partial}{\partial \rho}f)(0,t) dt.
 $$
 Then 
 $$
 \|(\Delta_\mathbb{H}+1)^{-1} h_2 \|_{L^2(\mathbb{H})}^2 \leq C_2 \| F_2 \|_{L^2(0,\ell)},
 $$
 where the constant $C_2$ is given by (\ref{const2}) and satisfies $C_2 \leq 1.61$
\end{theorem}
\begin{proof}
 As in the proof of the previous  theorem
 we can assume without loss of generality that the geodesic segment is part
 of the imaginary axis.
 As before we use Fermi coordinates
 $(\rho,t)$ so that the integral kernel is given by $k_s(\rho,t;\rho',t')=k_s(\rho,t-t';\rho',0)$.
 Thus,
 $$
   (\Delta_\mathbb{H}+3/2)^{-1} h_1 )(\rho,t)= \int F_1(t') \left(-\partial_{\rho'} k_s(\rho,t-t',\rho',0)\right)|_{\rho'=0} dt'.
 $$
 Now by the chain rule
 $$
  \left(-\partial_{\rho'} k_s(\rho,t-t',\rho',0)\right)|_{\rho'=0}  = (-\frac{\partial u}{\partial \rho} F'_s(u))(\rho,t-t',0,0)=
  \frac{ \sinh{\rho} }{2} F'_s(u(\rho,t-t',0,0)).
 $$
 Again, by the generalized Young inequality 
 $$
  \int |((\Delta_\mathbb{H}+3/2)^{-1} h_2 )(\rho,t)|^2 dt = \|F_1\|_{L^2}^2 (\frac{\sinh{\rho}}{2})^2 \left( \int | F'_s(u(\rho,t,0,0)) | dt \right)^2.
 $$
 Thus,
 $$
  \|(\Delta_\mathbb{H}+1)^{-1} h_2 \|_{L^2(\mathbb{H})}^2 \leq C_2 \| F_2 \|_{L^2(0,\ell)},
 $$
 where
 \begin{gather} \label{const2}
  C_2 = \left( \int \left( \int | F'_s(u(\rho,t,0,0)) | dt \right)^2 (\frac{\sinh{\rho}}{2})^2 \cosh{\rho} d\rho \right)^{\frac{1}{2}}
 \end{gather}
 Splitting the integral into a region $\cosh{\rho}>3$ and $\cosh{\rho}<3$ and using the estimates \ref{lem1}
 and \ref{lem2} one obtains
 \begin{gather*}
  C_2 \leq 
  \frac{\sqrt{3}}{16} \left(\left(27-16 \sqrt{2}\right) \left(5 \sqrt{2}-6 \pi +12
   \arctan\left(\sqrt{2}\right)\right)^2+48 \left(2 \sqrt{2}+\arccosh(3)\right)\right)^{1/2}. 
 \end{gather*}
 The numerical value of the right hand side of this inequality is about $1.6086$.
 \end{proof}

Both constants $C_1$ and $C_2$ can be computed using numerical integration and their values are about
$$
 C_1 \approx 0.313, \quad C_2 \approx 0.343,
$$
where we believe that at least two digits are correct. The corresponding estimates are non-rigorous however.

\section{Estimates on derivatives of the resolvent}

We will need the following bounds on the derivative of $k_{s}$ with respect to $s$.
\begin{lemma} \label{lemc1}
 Let $F_s(u)$ be the function defined in the previous section. Then, for any $\sigma>0$ we have the estimates
 \begin{gather*}
  |\frac{\partial}{\partial s} F_s(u)| \leq \tilde C_1(\sigma) \log(1+u) u^{-\sigma} + \tilde C_2(\sigma) u^{-\sigma},\\
  |\frac{\partial^2}{\partial u \partial s} F_s(u)| \leq  \tilde C_2(\sigma) u^{-\sigma-1} + |s| \left( \tilde C_1(\sigma) \log(1+u) +  \tilde C_2(\sigma) \right) u^{-\sigma-1},\\
  |\frac{\partial^3}{\partial u^2 \partial s} F_s(u)| \leq |2s+1| \tilde C_1(\sigma) u^{-\sigma-2} + 
  |s(s+1)| \left( \tilde C_1(\sigma) \log(1+u)+ \tilde C_2(\sigma) \right) u^{-\sigma-2},
 \end{gather*}
 where
 \begin{gather*}
  \tilde C_1(\sigma)= \frac{1}{4\pi} \int_0^1 (\xi(1-\xi))^{\sigma-1} d \xi,\\
  \tilde C_2(\sigma)= -\frac{1}{4\pi} \int_0^1 \log(\xi(1-\xi)) (\xi(1-\xi))^{\sigma-1} d \xi,
 \end{gather*}
 and $s=\sigma+\I r$.
 In particular $\tilde C_1(1/2)=\frac{1}{4}$ and $\tilde C_2(1/2)=\frac{\log(4)}{2}$.
\end{lemma}
\begin{proof}
 Simple differentiation shows that with $s=\sigma + \I r$:
 \begin{gather*}
  |\frac{\partial}{\partial s} F_{s }(u)| = |\frac{1}{4\pi} \int_0^1 \log{\frac{\xi(1-\xi)}{\xi+u}} (\xi (1-\xi))^{s-1} (\xi+u)^{-s} d \xi| \leq\\
  \leq -\frac{1}{4\pi} \int_0^1 \log(\xi(1-\xi)) (\xi(1-\xi))^{\sigma-1} d \xi \; u^{-\sigma} + \frac{1}{4\pi} \int_0^1 (\xi(1-\xi))^{\sigma-1} d \xi \; u^{-\sigma} \log(1+u), \end{gather*}
  which yields the first estimate. The second estimate and the third estimate are obtained in a similar manner from
   \begin{gather*}
  |\frac{\partial^2}{\partial u\partial s} F_{s}(u)| \leq \frac{1}{4 \pi}  
  \int_0^1  (\xi(1-\xi))^{\sigma-1} (\xi+u)^{-\sigma-1} d\xi \\ -\frac{|s|}{4\pi} \int_0^1 \log(\frac{\xi(1-\xi)}{\xi+u}) (\xi(1-\xi))^{\sigma-1} (\xi+u)^{-\sigma-1} d \xi.
  \end{gather*}
 and
 \begin{gather*}
  |\frac{\partial^3}{\partial u^2 \partial s} F_{s}(u)| \leq \frac{|2s+1|}{4 \pi}  
  \int_0^1  (\xi(1-\xi))^{\sigma-2} (\xi+u)^{-\sigma-1} d\xi \\ -\frac{|s(s+1)|}{4\pi} \int_0^1 \log(\frac{\xi(1-\xi)}{\xi+u}) 
  (\xi(1-\xi))^{\sigma-1} (\xi+u)^{-\sigma-2} d \xi.
  \end{gather*}
\end{proof}

Near zero one obtains similarly.

\begin{lemma} \label{lemc2}
 Let $F_s(u)$ be the function defined in the previous section. Then, for any $\sigma>0$ we have the estimates
 \begin{gather*}
  |\frac{\partial}{\partial s} F_{s}(u)| \leq \tilde C_3(\sigma),\\
  |u \frac{\partial^2}{\partial u \partial s} F_{s}(u)| \leq  \tilde C_4(\sigma) + \tilde C_3(\sigma)(r^2+\sigma^2)^{1/2},\\
  |u^2 \frac{\partial^3}{\partial u^2 \partial s} F_{s}(u)| \leq  | 2s+1 | \tilde C_4(\sigma) + \tilde C_3(\sigma)|s(s+1)|,
 \end{gather*}
 where $s=\sigma+\I r$.
\end{lemma}
In the important case $\sigma=1/2$ one obtains
\begin{gather*}
 \tilde C_3(1/2)=\frac{\pi}{4},\\
 \tilde C_4(1/2)=\sup_{u>0} \frac{1}{4 \pi}\int_{0}^1 \frac{u}{(\xi+u)^{3/2}} \frac{1}{\sqrt{\xi(1-\xi)}} d\xi.
\end{gather*}
The latter constant can numerically be calculated and is approximately given by
$$
 \tilde C_4(1/2) \approx 0.167878.
$$
%

\section{Resolvent estimates on hyperbolic cylinders} \label{resestcly}

Let  $Z$ be the hyperbolic cylinder obtained 
by factorizing the upper half space by the subgroup
of $SL(2,\mathbb{R})$ generated by the element $\left(\begin{matrix} \mathrm{e}^{\ell/2} & 0 \\ 0 & \mathrm{e}^{-\ell/2}\end{matrix}\right)$.
We use Fermi-coordinates $(\rho,t)$ which identifies the cylinder with $\mathbb{R} \times \mathbb{R}/ (\ell \mathbb{Z})$.
The integral kernel of the resolvent of the Laplace operator $\Delta_Z$ can be constructed from the resolvent
of $\Delta_{\mathbb{H}}$ by averaging over the group:
Let $x=(\rho,t)$ and $x'=(\rho',t')$ be distinct points on the hyperbolic cylinder and denote by $k_s(x,x')$
the integral kernel of $(\Delta_\mathbb{H}-s(1-s))^{-1}$. Then the integral kernel $k_s^Z(x,x')$
of $(\Delta_Z-s(1-s))^{-1}$ is given by
$$
 k_s^Z((\rho,t),(\rho',t')) = \sum_{n \in \mathbb{Z}} k_s((\rho,t),(\rho',t'+n \ell)) =\sum_{n \in \mathbb{Z}} k_s((\rho,t-t'),(\rho',n \ell))
$$
The sum converges absolutely whenever $\Re(s)>0$ because of the estimate in Lemma \ref{lem1}
and the fact that the distance between the points $(\rho,t)$ and $(\rho',t'+n \ell)$ grows like
$n \ell$ as $|n| \to \infty$.

Combining the estimates in Lemma \ref{lemc1} with Lemma \ref{lem3} we obtain

\begin{lemma} \label{greenestcyl}
 For any compact subset $M \subset Z$ we have on $M \times M$ the following estimates
 \begin{gather*}
  | k_s^Z(x,x')|   \leq \frac{1}{2 \pi} | \log{\rho(x,x')} | + C_1^M(\sigma),\\
  |\mathrm{grad}_{x'} k_s^Z(x,x')| \leq \frac{C_2^M(s)}{2 \pi} \frac{1}{\rho(x,x')} + C_3^M(s),
 \end{gather*}
 where $\rho(x,x') = 2 \,\mathrm{arcsinh}(u(x,x')^{1/2})$ is the distance between $x$ and $x'$
 and $\sigma=\Re(s)$.
\end{lemma}

Similarly Lemma \ref{lemc1} and Lemma \ref{lemc2}  imply the following estimates
\begin{lemma} \label{greenestcylder}
 Let $Z^L \subset Z$ be a truncated hyperbolic cylinder and let $d>0$. Then for
 all points $x \in Z^L$ and $x' \in \partial Z^L$ with $\rho(x,x')\geq d$ the following estimates hold.
 \begin{gather*}
  | \frac{\partial}{\partial s} k_{s}^Z(x,x') | \leq C_1^{M,d}(\sigma),\\
  | \frac{\partial}{\partial s}\mathrm{grad}_{x'} k_{s}^Z(x,x')| 
  \leq C_2^{M,d}(s),\\
   | \frac{\partial}{\partial s}\mathrm{grad}_{x'} \mathrm{grad}_{x} k_{s}^Z(x,x')| \leq C_3^{M,d}(s).
 \end{gather*}
\end{lemma}

Suppose now that $M \subset Z$ is a subset of the hyperbolic cylinder $Z$ whose boundary $\partial M$
is a finite union of geodesic segments. If $\psi$ is a function on $M$, smooth up the the boundary,
such that
$$
 (\Delta_Z-s(1-s)) \psi = 0
$$
in the interior of $M$.
Let $\phi_1$ and $\phi_2$ be the restriction of $\psi$ to the boundary and its outward normal
derivative respectively.
Of course, as functions on $Z$ we have
$$
 (\Delta_Z-s(1-s)) (\chi_M \psi) = h
$$ 
in the sense of distributions, where $h$ if the distribution defined as
$$
 h(f) := \int_{\partial M} (\partial_n f)(x) \phi_1(x) + f(x) \phi_2(x) dx.
$$
From this it follows that $\psi$ in the interior of $M$ can be expressed as
$$
 \psi=(\Delta_Z-s(1-s))^{-1} h.
$$
This is sometimes written in integral form
$$
 \psi(x) = \int_{\partial M} \partial_{n,x'} k_s^Z(x,x') \phi_1(x') + k_s^Z(x,x') \phi_2(x') dx'. 
$$
Using, Lemma \ref{greenestcyl} and applying the generalized Young inequality we obtain the
following estimate.

\begin{gather} \label{l2bound}
 \| \psi \|_{L^2(M)} \leq C_{M}(s) \left( \| \phi_1 \|_{L^2(\partial M)}^2 + \| \phi_2 \|_{L^2(\partial M)}^2 \right)^{1/2}.
\end{gather}

\section{Appendix-Estimates on the Counting Function}

Let $\mathrm{H}$ be the Heaviside step function defined by
$$
 \mathrm{H}(y) = \left \{ \begin{matrix} 1 & y>0 \\ 1/2  & y=0 \\ 0 & y<0 \end{matrix} \right.
$$
Suppose that $X$ is a compact connected oriented hyperbolic surface and 
let $(\phi_i)_{i \in \nz_0}$ be an orthonormal basis of eigenfunctions for the Laplace
operator on $X$. Let $$0=\lambda_0<\lambda_1\leq \lambda_2 \leq \ldots$$ be
the corresponding sequence of eigenvalues. Let $e(x,y,\tau)$ be integral kernel
of the operator $\mathrm{sign}(\tau) \mathrm{H}(\tau^2 - \Delta)$ and denote by $N_x(\tau)=e(x,x,\tau)$ be its restriction to the diagonal.
Of course, by the spectral theorem,
$$
 N_x(\tau)=\mathrm{sign}(\tau) \sum_{i} \mathrm{H}(\tau^2 - \lambda_i) | \phi_i(x) |^2.
$$
Integrating this over $X$ yields the eigenvalue counting function $N(\tau)$:
$$
 N(\tau)=  \mathrm{sign}(\tau) \sum_{i} \mathrm{H}(\tau^2 - \lambda_i).
$$
By construction $N_x$ and $N$ are non-decreasing odd functions on $\rz$.
Let $d(x)$ be twice the injectivity radius at the point $x \in X$. If
$L$ is the length of a systole we have the estimate $r(x) \leq L$.
Using Selberg's pretrace trace formula and finite propagation 
speed one can easily see that the cosine transform
of $N'_x(r)$ coincides on the interval $(-d(x),d(x))$ with the cosine transform of the function
$F'(r)$, where $F(t)$ is defined by
\begin{gather*}
 F(0)=0,\\
 F'(t)=2 \frac{1}{4 \pi} \mathrm{H}(t^2-1/4) | t | \tanh(\pi \sqrt{t^2-1/4}).
\end{gather*}
Note that since
$$
 \int_{1/2}^\infty 2  t (\tanh(\pi \sqrt{t^2-1/4})-1) dt=-\frac{1}{12}
$$
and $\tanh(\pi \sqrt{t^2-1/4})-1$ is non-positive on $[1/2,\infty)$
we obtain the estimate
$$
 -\frac{1}{3} \frac{1}{4 \pi} \leq \mathrm{sign}(t) G(t) \leq 0,
$$
where
 $G(t)=F(t) - \mathrm{sign}(t) \frac{1}{4 \pi} t^2$.
Let $\nu$ be the first Dirichlet eigenvalue of the operator $\frac{d^4}{dx^4}$ on the
interval $[-1/2,1/2]$ and let $\phi(x)$ be the corresponding normalized eigenfunction.
We think of $\phi$ as a function on $\rz$ by extending by zero.
It can be easily worked out that $\nu$ is the smallest non-zero solution to the equation
$\cosh(\lambda) \cos(\lambda)=1$ for $\lambda>0$. Its numerical value is 
$\nu \approx 4.73004074$. The estimate $\nu<5$ is an immediate consequence if the
intermediate value theorem.
As a test function in the Fourier Tauberian theorem we use the
function $\rho=(\hat \phi)^2$. Let $\rho_\delta$ and $\rho_{\delta,0}$
be defined as in \cite{MR1853753} as
\begin{gather*}
 \rho_\delta(r)=\delta \rho(\delta r),\\
 \rho_{\delta,0}(r)=\delta \rho_{1,0}(\delta r),
\end{gather*}
where
$$
 \rho_{1,0}(r) = \int_r^\infty t \rho(t) dt.
$$
Using $\rho_{\delta,0}(r)>0$ and $\|\rho_{\delta}\|_{L^1}=1$ we obtain
\begin{gather} \label{esteem1}
 (\rho_\delta * G)(r) \leq \| G \|_\infty = \frac{1}{12 \pi}\\  \label{esteem2}
  (\rho_{\delta,0} * G')(r) \leq 0.
\end{gather}
The Fourier Tauberian Theorem 1.3 in \cite{MR1853753} together with the estimates
Equations (2.9) and (2.10) in \cite{MR1853753} and our estimates (\ref{esteem1}) and (\ref{esteem2})
imply the following bounds of $N$.
\begin{theorem} \label{esth1}
 The counting function on a hyperbolic surface satisfies.
\begin{gather*}
 N_x(r) \leq \frac{1}{4\pi}\left( r^2 + \frac{4 \nu^2 +2 \nu \pi}{\pi d(x)} (r + \frac{\nu}{d(x)}) + \frac{1}{3}\right),\\
 N_x(r) \geq \frac{1}{4 \pi} \left( r^2 - \frac{4 \nu^2}{ \pi d(x)} (r+ \frac{\nu}{d(x)}) -\frac{1}{3}\right).
 \end{gather*}
\end{theorem}

Define the function $\tilde N_x(\lambda)=\mathrm{H}(\lambda) N_x(\sqrt{|\lambda|})$.
The local heat kernel trace $k_t(x)$ for $t>0$ may be defined as follows
$$
 k_t(x) = \sum_{\lambda_j \geq 0} \mathrm{e}^{-\lambda_j t} | \phi_i(x) |^2.
$$
If this sum is cut off at a point $c>0$ the remainder is given by
$$
 R^c_t(x) = \sum_{\lambda_j \geq c} \mathrm{e}^{-\lambda_j t} | \phi_i(x) |^2.
$$

Then of course integration by parts gives the following formula
$$
 R^c_t(x) =- \lim_{\epsilon \to 0^+} \tilde N_x(c + \epsilon) \mathrm{e}^{-c t} + t \int_c^\infty \tilde N_x(\lambda) \mathrm{e}^{-\lambda t} d \lambda.
$$
Combining this with the estimate from theorem \ref{esth1} we obtain
\begin{lemma} \label{heattrb}
The remainder $R^c_t(x)$ satisfies the estimate
\begin{gather*}
 4 \pi R^c_t(x) \leq \mathrm{e}^{-c t} \left( B_u(x) +B_l(x) + (A_u(x)+A_l(x)) \sqrt{c} +\frac{1}{t} \right)+\\+
 A_u(x)\sqrt{\frac{\pi}{4 t}} (1-\mathrm{erf}(\sqrt{c t})),
\end{gather*}
where
\begin{gather*}
 A_u(x)=\frac{4 \nu^2 + 2 \nu \pi}{\pi d(x)}, \quad
 B_u(x)=\frac{4 \nu^3 + 2 \nu^2 \pi}{\pi d(x)^2}+\frac{1}{3},\\
 A_l(x)=\frac{4 \nu^2}{\pi d(x)},\quad
 B_l(x)=\frac{4 \nu^3}{\pi d(x)^2}+\frac{1}{3}.
\end{gather*}
\end{lemma}

Setting $c=0$ one obtains a bound for the local heat kernel.
The heat trace
$$
 \tr(e^{-\Delta t}) = \sum_{\lambda_j\geq 0} \mathrm{e}^{-\lambda_j t}
$$
is obtained by integrating the local heat trace. Lemma \ref{heattrb}
therefore yields bounds on the heat trace in case $c=0$
and on
$$
 \sum_{\lambda_j \geq c} e^{-\lambda_j t} = \int_X R^c_t(x) dx.
$$
in general.

Theorem \ref{esth1} also immediately gives bounds on the eigenfunctions
\begin{corollary} \label{lbound}
 Let $\psi(x)$ be an eigenfunction of the Laplace operator with eigenvalue $\lambda$
 such that $\| \psi \|_{L^2}=1$.
 Then,
 $$
   | \psi(x) |^2 \leq \frac{1}{4 \pi} \left(\frac{8 \nu^2 +2 \nu \pi}{\pi d(x)} (\lambda^{1/2} + \frac{\nu}{d(x)}) + \frac{2}{3} \right).
 $$
\end{corollary}
For numerical purposes we will use the bound
$$
   | \psi(x) | \leq \sqrt{ \frac{11}{2 L}(\lambda^{1/2} + \frac{5}{L})+\frac{2}{3}}
 $$
 where $L$ is twice the radius of injectivity (which equals to the length of the shortest closed geodesic).

Similarly, one gets a bound on the derivative of the eigenfunctions in the following way. Fix a point $x \in X$
and a unit vector $\mathbf{n}$. As above we can define
$$
 N^1_x(\tau)=\mathrm{sign}(\tau) \sum_{i} \mathrm{H}(\tau^2 - \lambda_i) | \mathbf{n} \phi_i(x) |^2.
$$
Using the pre-trace formula and finite propagation speed it can easily worked out that
cosine transform
of $(N^1_x)'(r)$ coincides on the interval $(-d(x),d(x))$ with the cosine transform of the function
$\tilde F'(r)$, where $\tilde F(t)$ is defined by
\begin{gather*}
 \tilde F(0)=0,\\
 \tilde F'(t)= \frac{1}{4 \pi} \mathrm{H}(t^2-1/4) | t |^3  \tanh(\pi \sqrt{t^2-1/4}).
\end{gather*}
Using 
$$
 \int_{1/2}^\infty  t^3 (\tanh(\pi \sqrt{t^2-1/4})-1) dt=-\frac{17}{960}
$$
we obtain in the same way as before
$$
 -\frac{1}{30} \frac{1}{4 \pi} \leq \mathrm{sign}(t) \tilde G(t) \leq 0,
$$
where
 $\tilde G(t)=\tilde F(t) - \mathrm{sign}(t) \frac{1}{16 \pi} t^4$.

Again the Fourier Tauberian theorem in \cite{MR1853753} implies

\begin{theorem} \label{esth2}
The function $N^1_x$ satisfies
\begin{gather*}
 N^1_x(r) \leq \frac{1}{16\pi}  r^4 + \frac{2 \tilde \nu^2 +  \pi \tilde \nu}{4 \pi^2 d(x)} \left(r + \frac{\tilde \nu}{d(x)} \right)^3+ \frac{1}{8\pi}\frac{1}{15},\\
 N^1_x(r) \geq \frac{1}{16\pi} r^4 -  \frac{\tilde \nu^2}{2 \pi^2 d(x)} \left(r + \frac{\tilde \nu}{d(x)} \right)^3- \frac{1}{8\pi}\frac{1}{15},\\
 \end{gather*}
\end{theorem}
where $\tilde \nu$ is the same as $\nu_3$ in \cite{MR1853753} and satisfies the esimtate $\nu_3 \leq 6 \sqrt[6]{3} \leq 8$.

This results in the bound for the derivative of the eigenvalues
\begin{corollary} \label{lboundzwei}
 Let $\psi(x)$ be an eigenfunction of the Laplace operator with eigenvalue $\lambda$
 such that $\| \psi \|_{L^2}=1$ and let $n_x$ a unit tangent vector at $x \in X$.
 Then,
 $$
   |\mathbf{n} \psi(x) |^2 \leq \frac{4 \tilde \nu^2 +  \pi \tilde \nu}{4 \pi^2 d(x)} \left(r + \frac{\tilde \nu}{d(x)} \right)^3+ \frac{1}{4\pi}\frac{1}{15}.
 $$
\end{corollary}
Again, for numerical purposes we will simply use the bound
$$
   |\mathbf{n} \psi(x) | \leq \sqrt{
   \frac{6}{L} \left(r + \frac{8}{L} \right)^3+ \frac{1}{190} }.
$$

\noindent{\bf Acknowledgements.} 
We would like to thank Peter Buser for very interesting discussions and also for helping us
to compute the Fenchel Nielsen coordinates of surfaces given in various parametrizations. We are very
grateful to Ralf Aurich for providing the eigenvalues of surfaces for testing purposes. Our thanks also go to
Alex Barnett, Frank Steiner and Andreas Str\"ombergsson for useful comments and stimulating discussions.

\bibliographystyle{amsalpha}


\newcommand{\etalchar}[1]{$^{#1}$}
\providecommand{\bysame}{\leavevmode\hbox to3em{\hrulefill}\thinspace}
\providecommand{\MR}{\relax\ifhmode\unskip\space\fi MR }
\providecommand{\MRhref}[2]{%
  \href{http://www.ams.org/mathscinet-getitem?mr=#1}{#2}
}
\providecommand{\href}[2]{#2}

\end{document}